\documentclass[12pt,reqno]{amsart}
\usepackage{amsmath}
\usepackage{amssymb}
\usepackage{amstext}
\usepackage{mathrsfs}
\usepackage{a4wide}
\usepackage{graphicx}
\usepackage{bbm}
\usepackage{verbatim}
\usepackage{bm}
\usepackage[colorlinks, linkcolor=black, anchorcolor=black,
citecolor=black]{hyperref}
\usepackage{hyperref}
\hypersetup{hypertex=true,
            colorlinks=true,
            linkcolor=blue,
            anchorcolor=blue,
            citecolor=blue }

\allowdisplaybreaks \numberwithin{equation}{section}
\usepackage{color}
\usepackage{cases}

\numberwithin{equation}{section}

\newtheorem{theorem}{Theorem}[section]
\newtheorem{proposition}[theorem]{Proposition}
\newtheorem{corollary}[theorem]{Corollary}
\newtheorem{lemma}[theorem]{Lemma}
\theoremstyle{definition}
\newtheorem{remark}[theorem]{Remark}
\newtheorem*{conj}{Conjecture A}

\newtheorem{definition}[theorem]{Definition}

\theoremstyle{remark}

\newtheorem{example}[theorem]{Example}
\newtheorem{counterexample}[theorem]{Counterexample}

\def\d{\mathrm{d}}

\newcommand{\mv}{\mathbf{v}}
\newcommand{\R}{\mathbb{R}}
\newcommand{\LL}{\mathcal{L}}
\newcommand{\M}{\mathcal{M}}
\newcommand{\T}{\mathcal{T}}
\newcommand{\s}{\mathcal{S}(\R^N)}

\begin{document}
\title[On the rigidity of the 2D incompressible Euler equations ]{On the rigidity of the 2D incompressible Euler equations}
	
	\author{Yuchen Wang, Weicheng Zhan}
	
\address{School of Mathematical Science
Tianjin Normal University, Tianjin, 300074,
P.R. China}
\email{wangyuchen@mail.nankai.edu.cn}

\address{School of Mathematical Sciences, Xiamen University, Xiamen, Fujian, 361005, P.R. China}
\email{zhanweicheng@amss.ac.cn}


	\begin{abstract}
We consider rigidity properties of steady Euler flows in two-dimensional bounded domains. We prove that steady Euler flows in a disk with exactly one interior stagnation point and tangential boundary conditions must be circular flows, which confirms a conjecture proposed by F. Hamel and N. Nadirashvili in [J. Eur. Math. Soc., 25 (2023), no. 1, 323-368]. Moreover, for steady Euler flows on annuli with tangential boundary conditions, we prove that they must be circular flows provided there is no stagnation point inside, which answers another open problem proposed by F. Hamel and N. Nadirashvili in the same paper. We secondly show that the no-slip boundary conditions would result in absolute rigidity in the sense that except for the disks (\emph{resp}. annuli), there is no other smooth simply-connected (\emph{resp}. doubly-connected) bounded domain on which there exists a steady flow with only one (\emph{resp}. no) interior stagnation point and no-slip boundary conditions, and if present, the flow must be circular as well. The arguments are based on the geometry of streamlines of the flow and `local' symmetry properties for the non-negative solutions of semi-linear elliptic problems.
	\end{abstract}
	
	\maketitle{\small{\bf Keywords:} The incompressible Euler equations, Rigidity, Circular flows, Liouville theorem. \\

	
	\section{Introduction and Main results}
Consider stationary solutions of the two-dimensional incompressible Euler equations for an ideal fluid in a planar domain $D$
	\begin{align}\label{1-1}
		\begin{cases}
			\mathbf{v}\cdot \nabla \mathbf{v} =-\nabla P&\text{in}\ \, D,\\
			\nabla\cdot\mathbf{v}=0\,\ \, \ \ \ \ \ \ \  \ \, &\text{in}\ \, D,
\end{cases}
	\end{align}
where $\mathbf{v}=(v_1, v_2)$ is the velocity field and $P$ is the scalar pressure. Throughout this paper, the solutions $\mathbf{v}$ and $P$ are always understood in the classical sense, that is, they are (at least) of class $C^1$ and satisfy \eqref{1-1} everywhere in $D$. When the boundary $\partial D$ is not empty, a suitable boundary condition is needed such that the problem is well-posed. We shall explain this in detail later.

We are concerned in this paper with rigidity properties of two-dimensional steady Euler flows. Specifically, our goal is twofold, to investigate the effects of the boundary geometry on steady Euler flows, and to study the rigidity of steady fluids due to specific boundary conditions.

\subsection{Inheritance of the geometric symmetry of the domain}
We would like to address the following fundamental question pertaining to steady configurations
of fluid motion:

\smallskip
\emph{Given a domain $D$ with certain symmetry, to what extent must steady fluid states $\mathbf{v}$ inherit the geometric symmetry properties of the domain\,?}
\smallskip

It is a classical issue in fluid mechanics and received a lot of attention recently. An enormous amount of progress has been made, especially thanks to a series of works by Hamel and Nadirashvili. In \cite{HN3, HN1}, Hamel and Nadirashvili proved that in a strip, steady flows with no stagnation point and tangential boundary conditions is a parallel shear flow. The same conclusion also holds for a \emph{bounded} steady flow in a half-plane or the whole plane. Here the stagnation points of a flow $\mathbf{v}$ are the points $x$ for which $|\mv(x)|=0$. Recently, Hamel and Nadirashvili \cite{HN} considered steady Euler flows in bounded annuli, as well as in complements of disks, in punctured disks and in the punctured plane. They showed that if the flow does not have any stagnation point and satisfies the tangential boundary conditions together with further conditions at infinity in the case of unbounded domains and at the center in the case of punctured domains, then the flow is \emph{circular}, namely the streamlines are concentric circles.

  Besides these remarkable progress, it is worth pointing out that an important case is still left open, that is, the domain is a disk. Unlike the cases mentioned above, steady flows in a disk with tangential boundary conditions will possess at least one stagnation point, because the boundary of the disk is a closed streamline. It can be seen that all of the aforementioned known results strongly suggest that steady Euler flows would inherit the geometric symmetry of the domain well, and hence enjoy nice rigidity properties. Based on these facts, Hamel and Nadirashvili further proposed the following conjecture:

\begin{conj}[\cite{HN}, Conjecture 1.12]\label{con1}
  Let $D$ be an open disk centered at the origin. Let $z\in D$ and let $\mathbf{v}\in C^2(\overline{D}\backslash \{z\})$ be a bounded flow solving the Euler equations \eqref{1-1} with $D$ replaced by $D\backslash\{z\}$ and $\mathbf{v}\cdot \mathbf{n}=0$ on $\partial D$, where $\mathbf{n}$ denotes the outward unit normal on $\partial D$. Assume that $|\mathbf{v}|>0$ in $\overline{D}\backslash\{z\}$. Then $z$ is the origin and $\mathbf{v}$ is a circular flow.
\end{conj}

 Here a flow $\mathbf{v}$ is said to be circular (with respect to the origin) if $\mathbf{v}(x)$ is parallel to the vector $\mathbf{e}_\theta(x)=\left(-x_2/|x|, x_1/|x|\right)$ at every point $x\in D\backslash\{0\}$. We would like to point out that although this conjecture is quite natural, a rigorous proof seems highly non-trivial. The potential occurrence of singularities poses several challenges for mathematical analysis; see pages 332-333 in \cite{HN} for some discussions of  main difficulties.

 The first purpose of this paper is to provide an affirmative answer to Conjecture A under the additional condition that the vorticity $\omega:=\partial_1v_2-\partial_2v_1$ is Lebesgue integrable. More precisely, we will establish the following slightly stronger result.

\begin{theorem}\label{th1}
  Let $D$ be an open disk centered at the origin. Let $z\in D$ and let $\mathbf{v}\in C^2(\overline{D}\backslash \{z\})$ be a bounded flow solving the Euler equations \eqref{1-1} with $D$ replaced by $D\backslash\{z\}$ and $\mathbf{v}\cdot \mathbf{n}=0$ on $\partial D$, where $\mathbf{n}$ denotes the outward unit normal on $\partial D$. Assume that $|\mathbf{v}|>0$ in $D\backslash\{z\}$ and $\omega\in L^1(D)$. Then $z$ is the origin and $\mathbf{v}$ is a circular flow.
\end{theorem}

 Theorem \ref{th1} is a Liouville-type theorem for steady solutions of the incompressible Euler equations. Under the hypothesis of Theorem \ref{th1}, it follows that the potential isolated singularity can only occur at the center of the disk, and the steady flow exhibits (trivial) circular characteristics. There are many precise circular flows and we would like to give an example as follows.

\begin{example}\label{ex0-1}
  Consider the flow given by
  \begin{equation*}
   \mv(x)=\mathbf{e}_\theta(x),\ \ \ P(x)=\ln |x|.
  \end{equation*}
It is easy to check that $(\mv, P)$ satisfies the Euler equations \eqref{1-1} in the punctured unit disk. The vorticity of the flow is given by $\omega(x)=1/|x|$.
\end{example}

It is worth pointing out that the conclusion of Theorem \ref{th1} does not hold in general without the assumption that the flow $\mv$ is bounded. Otherwise, it would contradict with the following example. We begin with some notation. For $x\in \R^2$ and $r>0$,
\begin{equation*}
  B_r(x)=\left\{y\in \R^2: |y-x|<r \right\}
\end{equation*}
denotes the open disk with center $x$ and radius $r$. We shall write $B_r=B_r(0)$ for simplicity. Let $\mathbf{e}_r(x)=\left(x_1/|x|, x_2/|x|\right)$. We have the following counterexample (see also pages 229-330 in \cite{HN}):

\begin{counterexample}
  Consider the flow given by
  \begin{equation*}
   \mv(x)=\left[\left(1+\frac{1}{r^2}\right)\cos\theta \right]\mathbf{e}_\theta(x)+\left[\left(1-\frac{1}{r^2}\right)\sin\theta \right]\mathbf{e}_r(x),\ \ \ P(x)=-\frac{|\mv|^2}{2}
  \end{equation*}
  in the usual polar coordinates. The vorticity of the flow is identically zero. One can easily verify that $(\mv, P)$ satisfies the Euler equations \eqref{1-1} with tangential boundary conditions in the punctured unit disk $B_1\backslash\{0\}$. It has only two stagnation points in $\overline{B_1}\backslash\{0\}$ and they both lie on $\partial B_1$. However, $\mv$ is not a circular flow.
\end{counterexample}

The auxiliary assumption on the integrability of vorticity is technical, it is introduced to mitigate the singularity near the exceptional point of the flow and exclude bound flows with a highly oscillating profile. On the other hand, this condition also physically makes sense since the steady Euler flows could be extended onto the whole disk  in the sense of distribution. It is a very interesting and highly non-trivial question whether the integrability can be removed.

In particular, the following noteworthy result can be derived from Theorem \ref{th1}.

 \begin{corollary}\label{cor1}
  Let $D$ be an open disk centered at the origin. Let $\mathbf{v}\in C^2(\overline{D})$ satisfy the Euler equations \eqref{1-1} with $\mathbf{v}\cdot \mathbf{n}=0$ on $\partial D$, where $\mathbf{n}$ denotes the outward unit normal on $\partial D$. Assume that $z\in D$ is the only stagnation point of $\mathbf{v}$ in $D$, that is, $|\mathbf{v}(z)|=0$ and $|\mathbf{v}|>0$ in $D\backslash\{z\}$. Then $z$ is the origin and $\mathbf{v}$ is a circular flow.
 \end{corollary}

 Corollary \ref{cor1} shows that \emph{any} smooth steady Euler flow in disks with exactly one interior stagnation point must be a circular flow, isolating such configurations from non-circular steady states. This addresses an intriguing open issue proposed in \cite{HN} (see also the discussion in \S 3.1 of the survey paper \cite{Dri}). Below is an example of such flow configurations.
\begin{example}\label{ex0}
  Consider the smooth flow given by
  \begin{equation*}
   \mv(x)=2\sin(|x|^2)(x_2, -x_1),\ \ \ P(x)=-2\sin^2(|x|^2)|x|^2-F(\cos(|x|^2)),
  \end{equation*}
  where
  \begin{equation*}
     F(s) =4\int_{0}^{s}\left(t\arccos t+\sqrt{1-t^2}\right)\d t.
  \end{equation*}
  It is easy to check that $(\mv, P)$ satisfies the Euler equations \eqref{1-1} with tangential boundary conditions in $B_{\sqrt{\pi}}$.
\end{example}


It is known that in a bounded convex domain, the uniqueness of the interior stagnation point is satisfied if the flow is stable in the sense of Arnold; see \cite{Nad}. An interesting issue is to determine the possible location of the unique interior stagnation point (in the ansatz). Our next result shows that if the domain is \emph{Steiner symmetric} with respect to a certain direction, then the unique interior stagnation point of a steady flow can only lie on the corresponding symmetry line. Here a domain is referred to as Steiner symmetric with respect to a given direction if it is convex in that direction and reflectionally symmetric about a line orthogonal to that direction.

\begin{theorem}\label{th11}
   Let $D$ be a $C^2$ non-empty simply connected bounded domain of $\mathbb{R}^2$, such that it is Steiner symmetric with respect to the direction $\mathbf{e}_1=(1, 0)$ and the symmetry line is given by $\Gamma_0=\left\{x\in \R^2: x_1=0 \right\}$. Let $\mathbf{v}\in C^2(\overline{D})$ satisfy the Euler equations \eqref{1-1} with $\mathbf{v}\cdot \mathbf{n}=0$ on $\partial D$, where $\mathbf{n}$ denotes the outward unit normal on $\partial D$.

   Assume that $\mv$ has only one interior stagnation point $z \in D$. Then $z\in \Gamma_0$.
\end{theorem}

For example, if $D$ is an ellipse, then the unique interior stagnation point can only be located at the center, since it is the only point where two symmetrical lines intersect. In particular, in the case of a disk, the unique interior stagnation point must be the center of the disk, which is consistent with the result of Theorem \ref{th1}.


Next, let us now consider the case of bounded two-dimensional annuli. We would like to address another open problem proposed in \cite{HN}, which is concerned with the rigidity of steady Euler flows in annuli. Let us start with some notation.

For $0<a<b<+\infty$, let $\Omega_{a, b}=\{x\in \mathbb{R}^2: a<|x|<b\}$, $C_a=\{x\in \mathbb{R}^2: |x|=a\}$ and $C_b=\{x\in \mathbb{R}^2: |x|=b\}$. In \cite{HN}, Hamel and Nadirashvili established the following result.
\begin{theorem}[\cite{HN}, Theorem 1.2]
  Let $D=\Omega_{a, b}$. Let $\mv\in C^2(\overline{\Omega_{a, b}})$ satisfy the Euler equations \eqref{1-1} with $\mathbf{v}\cdot \mathbf{n}=0$ on $\partial D$, where $\mathbf{n}$ denotes the outward unit normal on $\partial D$. Assume moreover that
  \begin{equation}\label{1-3}
    \left\{x\in \overline{\Omega_{a, b}}: |\mv(x)|=0 \right\}\subsetneq C_a\ \ \ \text{or}\ \ \     \left\{x\in \overline{\Omega_{a, b}}: |\mv(x)|=0 \right\}\subsetneq C_b.
  \end{equation}
  Then $|\mv|>0$ in $\overline{\Omega_{a, b}}$ and $\mv$ is a circular flow, and there is a $C^2([a, b])$ function $V$ with constant strict sign such that
  \begin{equation*}
    \mv(x)=V(|x|)\mathbf{e}_\theta(x)\ \ \ \text{for all }\, x\in \overline{\Omega_{a, b}}.
  \end{equation*}
\end{theorem}

Note that if $\mv$ has no stagnation point in $\overline{\Omega_{a, b}}$, then the condition \eqref{1-3} is automatically fulfilled. In other words, a steady flow having no stagnation point in the \emph{closed} annulus $\overline{\Omega_{a, b}}$ must be a circular flow; see Theorem 1.1 in \cite{HN}. Nevertheless, the condition \eqref{1-3} seems unnatural and should be due to technical reasons, which unfortunately excludes many interesting flows. The following two examples escape the scope of this theorem.

\begin{example}\label{ex1}
  The smooth flow given by
  \begin{equation*}
    \mv(x)=(|x|-a)\mathbf{e}_\theta(x),\ \ \ P(x)=|x|^2/2-2a|x|+a^2\ln|x|
  \end{equation*}
clearly solves \eqref{1-1} with $D=\Omega_{a, b}$ and satisfies $\mathbf{v}\cdot \mathbf{n}=0$ on $\partial D$. Notice that $|\mv|=0$ on $C_a$, so it does not satisfy the condition \eqref{1-3}.
\end{example}

\begin{example}\label{ex1-1}
    Consider the smooth flow given by
  \begin{equation*}
   \mv(x)=4(|x|-1)(|x|-2)(x_2, -x_1),\ \ \ P(x)=-\frac{|\mv|^2}{2}-F\left(|x|^2(2-|x|)^2\right),
  \end{equation*}
  where
  \begin{equation*}
     F(s) =4\int_{0}^{s}\left(4\sqrt{t}+\sqrt{1-\sqrt{t}}-3 \right)\d t.
  \end{equation*}
  We check that $(\mv, P)$ solves the Euler equations \eqref{1-1} with $D=\Omega_{1, 2}$. Note that $\mathbf{v}=0$ on $\partial D$ and hence the condition \eqref{1-3} is also violated.
\end{example}

With this in mind, Hamel and Nadirashvili proposed the following open question (see page 349 of \cite{HN}):

\begin{center}
  \emph{if $|\mv|>0$ in $\Omega_{a, b}$, then $\mv$ is a circular flow\,?}
\end{center}

Our second main result gives a positive answer to this question. Specifically, we have
\begin{theorem}\label{th2}
    Let $D=\Omega_{a, b}$. Let $\mv\in C^2(\overline{\Omega_{a, b}})$ satisfy the Euler equations \eqref{1-1} with $\mathbf{v}\cdot \mathbf{n}=0$ on $\partial D$, where $\mathbf{n}$ denotes the outward unit normal on $\partial D$. If $|\mv|>0$ in $\Omega_{a, b}$, then $\mv$ is a circular flow.
\end{theorem}
Theorem \ref{th2} shows that \emph{any} smooth steady solution of the Euler equations in an annulus which never vanishes \emph{inside} the domain must be a circular flow, isolating such configurations from non-circular steady states.

\subsection{Rigid effects induced by the boundary conditions}
The second part of this paper focuses on rigid effects induced by the boundary conditions. Let us assume that $D \subset \mathbb{R}^2$ is a bounded domain with a smooth boundary. The most common boundary condition is the \emph{tangential} boundary condition, namely, $\mathbf{v}$ is (at least) continuous up to the boundary and tangential there:
\begin{equation}\label{1-4}
  \mathbf{v}\cdot\mathbf{n}=0\ \ \ \text{on}\ \, \partial D,
\end{equation}
where $\mathbf{n}$ is the outward unit normal on $\partial D$. All of the aforementioned results work with this boundary condition. We are interested in another important boundary condition, the so-called \emph{no-slip} boundary condition, that is,
\begin{equation}\label{1-5}
  \mathbf{v}=0\ \ \ \text{on}\ \, \partial D.
\end{equation}
Note that both the steady flows given in Examples \ref{ex0} and \ref{ex1-1} satisfy the no-slip boundary condition \eqref{1-5}. From a physical point of view, the no-slip boundary condition requires that the particles of the fluid “adhere” to the boundary. Clearly, it is much stronger than the tangential boundary condition \eqref{1-4}. This adherence condition is reasonable for the Navier-Stokes equations describing the motion of viscous fluids (see \cite{Gal, Lions}). However, it seems inappropriate for the Euler equations, because in general, the no-slip condition over-determines the fluid flow problem so that no solution to the problem exists. We refer the reader to \cite{Day} and the references therein for related comments and discussions. It is noteworthy that the existence of a solution to Leray's problem for the stationary Navier-Stokes equations can be linked to the nonexistence of nontrivial steady Euler flows under the no-slip boundary condition and additional conditions; see \cite{Kor}.

The second part of this paper aims to establish some rigidity results concerning steady Euler flows under the no-slip boundary condition. Roughly speaking, we will show that the no-slip boundary condition leads to absolute rigidity, namely non-trivial steady solution merely exist in some sufficiently nice domains.

Our third main result deals with the case when the fluid occupies a simply connected domain.
\begin{theorem}\label{th3}
  Let $D$ be a $C^2$ non-empty simply connected bounded domain of $\mathbb{R}^2$. Let $\mv\in C^2(\overline{D})$ satisfy the Euler equations \eqref{1-1} and assume that $\mv=0$ on $\partial D$. Assume moreover that $\mv$ has a unique stagnation point in $D$. Then, up to translation,
  \begin{equation*}
    D=B_R
  \end{equation*}
  for some $R>0$. Furthermore, the unique stagnation point of $\mv$ is the center of the disk and $\mv$ is a circular flow.
\end{theorem}

\begin{example}\label{ex2}
  Consider the smooth flow given in Example \ref{ex0}, namely,
  \begin{equation*}
   \mv(x)=2\sin(|x|^2)(x_2, -x_1),\ \ \ P(x)=-2\sin^2(|x|^2)|x|^2-F(\cos(|x|^2)),
  \end{equation*}
  where
  \begin{equation*}
     F(s) =4\int_{0}^{s}\left(t\arccos t+\sqrt{1-t^2}\right)\d t.
  \end{equation*}
  Clearly, $(\mv, P)$ solves the Euler equations \eqref{1-1} with the no-slip boundary condition \eqref{1-5} in $B_{\sqrt{\pi}}$.
\end{example}
Theorem \ref{th3} and Example \ref{ex2} combine to tell us that, if $D$ is a smooth simply connected bounded domain, then there is no steady flow in $D$ with a unique interior stagnation point and no-slip boundary conditions, unless $D$ is a disk.

Our next result involves the case of fluid occupying a doubly connected domain.
\begin{theorem}\label{th4}
  Let $\Omega_1$ and $\Omega_2$ be two $C^2$ non-empty simply connected bounded domains of $\mathbb{R}^2$ such that $\overline{\Omega_1}\subset \Omega_2$, and denote
  \begin{equation*}
    D=\Omega_2\backslash \overline{\Omega_1}.
  \end{equation*}
  Let $\mv\in C^2(\overline{D})$ satisfy the Euler equations \eqref{1-1} with $\mv=0$ on $\partial D$. Assume that $|v|>0$ in $D$. Then $\Omega_1$ and $\Omega_2$ are two concentric disks and, up to translation,
  \begin{equation*}
    D=  \Omega_{a,b}
  \end{equation*}
  for some $0<a<b<\infty$ and $\mv$ is a circular flow.
\end{theorem}

\begin{example}\label{ex3}
    Consider the smooth flow given by
  \begin{equation*}
   \mv(x)=2\sin(|x|^2)(x_2, -x_1),\ \ \ P(x)=-2\sin^2(|x|^2)|x|^2-F(\cos(|x|^2)),
  \end{equation*}
  where
  \begin{equation*}
     F(s) =4\int_{0}^{s}\left(t(\arccos (-t)+\pi)-\sqrt{1-t^2}\right)\d t.
  \end{equation*}
  One can easily verify that $(\mv, P)$ solves the Euler equations \eqref{1-1} with the no-slip boundary condition \eqref{1-5} in $\Omega_{\sqrt{\pi}, \sqrt{2\pi}}$.
\end{example}

From Theorem \ref{th4} and Example \ref{ex3}, we see that if $D$ is a smooth doubly connected bounded domain, then steady Euler flows in $D$ with no-slip boundary conditions must vanish at some interior points, unless $D$ is an annulus.

\begin{remark}
  It should be mentioned that Hamel and Nadirashvili have considered similar issues; see Theorems 1.10 and 1.13 in \cite{HN}. In their results, the steady flow $\mv$ only needs to satisfy tangential boundary conditions, but $|\mv|$ must be \emph{positive} constant on the boundary. Recently, Ruiz \cite{Ruiz} established some similar symmetry results for compactly supported steady solutions of the Euler equations in the whole plane. In the case considered by \cite{Ruiz}, the vorticity of the fluid will automatically vanish on the free boundary, except for Theorem C. However, Theorem C in \cite{Ruiz} was obtained with auxiliary condition $\partial_{\mathbf{n}}\mv\not= 0$ on the boundary which seems unnatural. Our Theorems \ref{th3} and \ref{th4} can be taken as a complement to these works. We also note that the conclusion does not hold in general without the uniqueness or nonexistence of the interior stagnation point; see e.g. \cite{Ago, En1, En2, Kam, Ruiz1}.
\end{remark}

Two more general forms of Theorems \ref{th3} and \ref{th4} are described below. First, by combining Theorem \ref{th3} with Theorem 1.10 of \cite{HN}, we immediately get the following result.
\begin{theorem}\label{th5}
  Let $D$ be a $C^2$ non-empty simply connected bounded domain of $\mathbb{R}^2$. Let $\mv\in C^2(\overline{D})$ satisfy the Euler equations \eqref{1-1} and assume that $\mathbf{v}\cdot\mathbf{n}=0$, and that $|\mv|$ is constant on $\partial D$. Assume moreover that $\mv$ has a unique stagnation point in $D$. Then, up to translation,
  \begin{equation*}
    D=B_R
  \end{equation*}
  for some $R>0$. Furthermore, the unique stagnation point of $\mv$ is the center of the disk and $\mv$ is a circular flow.
\end{theorem}

In view of Theorem \ref{th4} above and Theorem 1.13 of \cite{HN}, one may expect to establish a unified result. In fact, we can further prove the following general result.
\begin{theorem}\label{th6}
  Let $\Omega_1$, $\Omega_2$ and $D$ be as in Theorem \ref{th4}. Let $\mv\in C^2(\overline{D})$ satisfy the Euler equations \eqref{1-1} with $\mathbf{v}\cdot\mathbf{n}=0$ on $\partial D$, and assume that $|\mv|$ is constant on $\partial \Omega_1$ and on $\partial \Omega_2$. Assume further that $|v|>0$ in $D$. Then $\Omega_1$ and $\Omega_2$ are two concentric disks and, up to translation,
  \begin{equation*}
    D=\left\{x\in \mathbb{R}^2: a<|x|<b\right\}
  \end{equation*}
  for some $0<a<b<\infty$ and $\mv$ is a circular flow.
\end{theorem}

Theorems \ref{th5} and \ref{th6} can be viewed as two classification results for the steady Euler equations in simply or doubly connected bounded domains with free boundaries.

\subsection{Sketch of the proof}
Let us now sketch the idea of the proof. The proof is based on the study of the geometric properties of the streamlines of the flow and on `local' symmetry properties for the non-negative solutions of semi-linear elliptic equations with a continuous nonlinearity. There are two major ingredients. The first is to show that the corresponding stream function solves a certain semi-linear elliptic problem, which is based on the study of the geometric properties of the streamlines of the flow. This step can be done by arguments similar to the ones used in \cite{HN}; see Proposition \ref{pro1} below. The second is to prove symmetry properties of non-negative solutions to the semi-linear elliptic equation, which leads to the desired conclusion. This step is the hard part. Relevant results in \cite{HN, Ruiz} are directly based on the well-known technique of moving planes, which goes back to  Alexandrov for the study of manifolds with constant mean curvature \cite{Ale}. The moving plane method is a very powerful technique in proving symmetry results for positive solutions of elliptic and parabolic problems in symmetric domains. We refer to the survey article \cite{Ni} and the references therein for more extensive discussions on this and related topics. However, we have to point out that the moving plane method fails in general for a non-Lipschitz continuous nonlinearity (see \cite{Bro1} for an example). More precisely, in this case, the moving plane method is often applicable only under some additional conditions (see \cite{Ruiz} for some further comments). In our situation, the nonlinearity in general is merely continuous, far from being Lipschitz continuous (see e.g. Examples \ref{ex2-1} and \ref{ex2-2} below). Therefore, it seems unlikely that the moving plane method can be used directly to obtain the desired conclusion. Nevertheless, we mention that partial symmetry can still be achieved using an incomplete version of the moving plane method (see some comments on pages 332-333 of \cite{HN}). On the other hand, Brock \cite{Bro0, Bro1} has developed a rearrangement technique called continuous Steiner symmetrization to show some `local' kind of symmetry for the non-negative solutions of semi-linear elliptic problems with continuous nonlinearity (see also Section \ref{s5} below). We are able to use the rearrangement technique to get some `local' kind of symmetry of the solutions, despite the possible lack of regularity of the nonlinearity. We would like to mention that similar continuous Steiner symmetrization techniques were recently used to prove the symmetry of equilibrium configurations for nonlinear aggregation-diffusion equations; see \cite{Car}.

Our strategy is to effectively combine the moving plane method and the continuous Steiner symmetrization techniques. Taking into account the assumptions about the flow, we can further show that `local' symmetry (achieved through continuous Steiner symmetrization techniques) and partial symmetry (achieved through the moving plane method) suffice to establish the desired symmetry of the flow under consideration. We also remark that the proof of Theorem \ref{th6} further requires elaborate techniques to establish some new Serrin-type results; see Section \ref{s5} for more details.

Finally, we also would like to mention that in recent years, there has been tremendous interest in investigating the rigidity and flexibility in steady fluid motion; see \cite{Cons, Coti, Gom, Gui, HK, HN3, HN2, HN1, HN, Li, Ruiz} and the references therein. We refer the interested reader to \cite{HK} for some relatively complete comments on recent rigidity and Liouville-type results for the Euler or Navier-Stokes equations.

The rest of the paper is organized as follows: In Section \ref{s2}, we give some preliminary results which will be used in the sequel. The proofs of Theorems \ref{th1} and \ref{th2} are presented in Section \ref{s3}. In Section \ref{s4}, we give the proofs of Theorems \ref{th3} and \ref{th4}. We devote Section \ref{s5} to proving Theorem \ref{th6}. The proof of Theorem \ref{th11} will be postponed until Section \ref{s6} because it required some results established earlier.

\section{Preliminary results}\label{s2}
In this section, we prove several preliminary results which will be frequently used in the proofs of main results. The first two assertions connect steady Euler flows with semi-linear elliptic equations of the stream function.

\begin{proposition}\label{pro1}
Let $D$ be a $C^2$ non-empty simply connected bounded domain of $\mathbb{R}^2$. Let $z\in D$ and let $\mathbf{v}\in C^2(\overline{D}\backslash \{z\})$ be a bounded flow solving the Euler equations \eqref{1-1} with $D$ replaced by $D\backslash\{z\}$ and $\mathbf{v}\cdot \mathbf{n}=0$ on $\partial D$, where $\mathbf{n}$ denotes the outward unit normal on $\partial D$. Assume that $|\mathbf{v}|>0$ in $D\backslash\{z\}$.

Then there is a $C^3(\overline{D}\backslash\{z\})\cap W_0^{1,\infty}(D)$ scalar function $u$ such that
  \begin{equation*}
  \nabla^\perp u=\mathbf{v},\ \ \ \text{that is},\ \ \ \partial_1 u=v_2\ \ \text{and}\ \ \partial_2u=-v_1
\end{equation*}
in $D\backslash \{z\}$. Moreover, there exists a function $f\in C(\overline{I}\backslash\{u(z)\})\cap C^1(I)$ with $I=\{u(x): x\in D\backslash\{z\}\}$, such that
  \begin{equation}\label{2-1}
   \Delta u+f(u)=0\ \ \ \text{in}\ \ D\backslash\{z\}.
  \end{equation}
  If, additionally, $\{x\in \partial D: |\mathbf{v}(x)|=0\}\subsetneq \partial D$, then $f\in C^1(\overline{I}\backslash\{u(z)\})$.
\end{proposition}

\begin{proof}
This result is essentially included in \cite{HN} (see also \cite{Ruiz}), but we would like to summarize the argument here for the sake of completeness.  The existence of a stream function $u\in C^3(\overline{D}\backslash\{z\})\cap W_0^{1,\infty}(D)$ is a consequence of the fact that $\mv$ is bounded, divergence-free and satisfies the tangential boundary condition \eqref{1-4}. It remains to show \eqref{2-1}.

By assumption, the function $u$ has a unique critical point $z\in D$. We can assume, without loss of generality, that (after possibly changing $\mathbf{v}$ into $-\mathbf{v}$ and $u$ into $-u$)
\begin{equation*}
  0<u(x)<u(z)\ \ \ \text{for all}\ \ x\in D\backslash\{z\}.
\end{equation*}
For any given point $y\in D\backslash\{z\}$, let $\sigma_y$ be the solution of
	\begin{align}\label{2-3}
		\begin{cases}
			 \dot{\sigma}_y(t) = \nabla u(\sigma_y(t)), &\\
             \sigma_y(0)=y.\ \ &
\end{cases}
	\end{align}
Then by Lemma 2.2 in \cite{HN}, there are quantities $t_y^{\pm}$ such that $-\infty\le t^-_y<0<t^+_y\le +\infty$ and the solution $\sigma_y$ of \eqref{2-3} is of class $C^1((t_y^-, t^+_y))$ and ranges in $D\backslash \{z\}$, with
	\begin{align}\label{2-4}
		\begin{cases}
		\text{dist}\left(\sigma_y(t), \partial D \right)\to 0\ \,\text{and}\ \,u(\sigma_y(t))\to 0\ \,\text{as}\ \,t\to t_y^-,\ \ &\\
|\sigma_y(t)-z|\to 0 \ \,\text{and}\ \,u(\sigma_y(t))\to u(z)\ \,\text{as}\ \, t\to t_y^+. &
\end{cases}
	\end{align}
Note that the function $g:=u\circ \sigma_y\in C^1((t_y^-, t_y^+))$ is increasing since $(u\circ \sigma_y)'(t)=|\nabla u(\sigma_y(t))|^2=|\mathbf{v}(\sigma_y(t))|^2>0$ for all $t\in (t_y^-, t_y^+)$. So $g$ is a strictly increasing diffeomorphism from $(t_y^-, t_y^+)$ onto $(0, u(z))$. Consider the function $f: (0, u(z))\to \mathbb{R}$ defined by
\begin{equation}\label{2-5}
  f(\tau)=-\Delta u(\sigma_y(g^{-1}(\tau)))\ \ \ \text{for}\ \ \tau\in (0, u(z)).
\end{equation}
Then $f$ is of class $C^1((0, u(z)))$ by the chain rule. The equation $\Delta u+f(u)=0$ is now satisfied along the curve $\sigma_y((t_y^-, t_y^+))$. Let us check it in the whole set $D\backslash\{z\}$. Consider any point $x\in D\backslash\{z\}$. Let $\xi_x$ be the solution of
	\begin{align*}
		\begin{cases}
			 \dot{\xi}_x(t) = \mathbf{v}(\xi_x(t)), &\\
             \xi_x(0)=x.\ \ &
\end{cases}
	\end{align*}
By Lemma 2.6 in \cite{HN}, we see that $\xi_x$ is defined in $\mathbb{R}$ and periodic. Furthermore, the streamline $\Im_x:=\xi_x(\mathbb{R})$ is a $C^2$ Jordan curve surrounding $z$ in $D$ and meets the curve $\sigma_y((t_y^-, t_y^+))$ once. Hence, there is $s\in (t_y^-, t_y^+)$ such that $\sigma_y(s)\in \Im_x$. On the one hand, the stream function $u$ is constant along the streamline $\Im_x$. On the other hand, it follows from the Euler equations \eqref{1-1} that $\mv \cdot \nabla(\Delta u)=0$ in $D\backslash\{z\}$, and hence the vorticity $\Delta u$ is constant along the streamline $\Im_x$ too. Therefore, \eqref{2-5} implies
\begin{equation*}
  \Delta u(x)+f(u(x))=\Delta u(\sigma_y(s))+f(u(\sigma_y(s)))=\Delta u(\sigma_y(s))+f(g(s))=0.
\end{equation*}
Hence, $\Delta u+f(u)=0$ in $D\backslash\{z\}$. By the continuity of $u$ in $\overline{D}$, we have
\begin{equation*}
 \max_{t\in \mathbb{R}}{\text{dist}(\xi_x(t),\partial D)} \to 0  \ \text{as}\ \text{dist}(x, \partial D)\to 0\ \ \ \text{and}\ \ \   \max_{t\in \mathbb{R}}|\xi_x(t)-z|\to 0\ \text{as}\ |x-z|\to 0.
\end{equation*}
Since $\Delta u$ is continuous in $\overline{D}\backslash\{z\}$ and constant along any streamline of the flow, we see that $\Delta u$ is constant on $\partial D$. Call $d$ the value of $\Delta u$ on $\partial D$. Set $f(0)=-d$. Then we infer from \eqref{2-4} and \eqref{2-5} that $f:[0, u(z))\to \mathbb{R}$ is continuous in $[0, u(z))$ and that the equation $\Delta u+f(u)=0$ actually holds in $\overline{D}\backslash\{z\}$. Finally, if $\{x\in \partial D: |\mathbf{v}(x)|=0\}\subsetneq \partial D$, then by Lemma 2.8 in \cite{HN}, we see that $f\in C^1(\overline{I}\backslash\{u(z)\})$. The proof is thus complete.
\end{proof}

\begin{example}\label{ex2-1}
  Let $\mv$ be the steady flow in Example \ref{ex0-1}, i.e., $\mv(x)=\mathbf{e}_\theta(x)$. Then
  \begin{equation*}
\begin{split}
   u(x) & =|x|-1, \\
    I & = \left\{u(x)\mid x\in B_1\backslash\{0\}\right\}=(-1, 0).
\end{split}
  \end{equation*}
  Let $f(s)  =-1/(s+1),\ s\in (-1, 0]$. One can easily verify that the equation $\Delta u+f(u)=0$ holds in $\overline{B_1}\backslash\{0\}$.
\end{example}

Similar results hold for the case of doubly connected domains; see \cite{HN,Ruiz}.

\begin{proposition}\label{pro1-2}
 Let $\Omega_1$ and $\Omega_2$ be two $C^2$ non-empty simply connected bounded domains of $\mathbb{R}^2$ such that $\overline{\Omega_1}\subset \Omega_2$, and denote
  \begin{equation*}
    D=\Omega_2\backslash \overline{\Omega_1}.
  \end{equation*}
   Let $\mv\in C^2(\overline{D})$ satisfy the Euler equations \eqref{1-1} with $\mathbf{v}\cdot\mathbf{n}=0$ on $\partial D$. Assume that $|v|>0$ in $D$. Then there is a $C^3(\overline{D})$ scalar function $u$ such that
  \begin{equation*}
  \nabla^\perp u=\mathbf{v},\ \ \ \text{that is},\ \ \ \partial_1 u=v_2\ \ \text{and}\ \ \partial_2u=-v_1,
\end{equation*}
and $u$ is constant on each connected component of $\partial D$. Moreover, there exists a function $f\in C(\overline{I})\cap C^1(I)$ with $I=\{u(x): x\in D\}$, such that
  \begin{equation}\label{2-1-1}
   \Delta u+f(u)=0\ \ \ \text{in}\ \ D.
  \end{equation}
\end{proposition}

\begin{example}\label{ex2-2}
  Let $\mv$ be the steady flow in Example \ref{ex1}, i.e., $ \mv(x)=(|x|-a)\mathbf{e}_\theta(x)$. Then
  \begin{equation*}
\begin{split}
   u(x) & =(|x|-a)^2/2\ \ \ \text{(up to additive constants)}, \\
    I & = \{u(x)\mid x\in {\Omega_{a, b}}\}=(0, (b-a)^2/2).
\end{split}
  \end{equation*}
  Let $f(s)  =-2+a/(a+\sqrt{2s}),\ s\in \overline{I}$. One can easily verify that the equation $\Delta u+f(u)=0$ holds in $\overline{\Omega_{a, b}}$.
\end{example}

For convenience, we denote by $Q_r(z)$ the closed disk in $\mathbb{R}^2$, with radius $r>0$ centered at $z\in \R^2$, and $Q_r:=Q_r(0)$. Inspired by Brock \cite{Bro1}, we introduce the following `local' kind of symmetry.
\begin{definition}\label{de1}

  Let $u\in C(\R^2)$ be a non-negative function with compact support. Suppose that $u$ is continuously differentiable on $V:=\{x\in \R^2: 0<u(x)<\sup_{\R^2} u\}$. We say that $u$ is \emph{locally symmetric} if
    \begin{itemize}
    \item [(1)]$\displaystyle  V=\bigcup_{k\in K} A_k \cup \{x\in V: \nabla u(x)=0\}$, \text{where}
    \begin{equation*}
      A_k=B_{R_k}(z_k)\backslash {Q_{r_k}(z_k)},\ \ \ z_k\in \R^2,\ \ \ 0\le r_k<R_k;
    \end{equation*}
     \item[(2)]$K$ is a countable set;
        \smallskip
    \item [(3)] the sets $A_k$ are pairwise disjoint;
    \smallskip
     \item [(4)]$u(x)=U_k(|x-z_k|)$, $x\in A_k$, where $U_k\in C^1([r_k, R_k])$;
         \smallskip
    \item [(5)]$U'_k(r)<0$ for $r\in (r_k, R_k)$;
        \smallskip
    \item [(6)]$u(x)\ge U_k(r_k),\ \forall\,x\in B_{r_k}(z_k)$, $k\in K$.
  \end{itemize}
\end{definition}

The following result concerns the local symmetry of non-negative solutions to semi-linear elliptic equations with a nonlinearity of low regularity, which follows from Theorems 6.1 and 7.2 in \cite{Bro1}.
\begin{proposition}[\cite{Bro1}]\label{pro2}
Let $R>0$. Let $f\in L^\infty([0, +\infty))$ be such that
\begin{equation*}
  f=f_1+f_2+f_3,
\end{equation*}
where $f_1$ is a continuous function, $f_2$ is a non-increasing function, and $f_3$ is a non-decreasing function. Let $u\in H^1_0(B_R)\cap C(\overline{B_R})$ be a weak solution of the following problem
	\begin{align*}
		\begin{cases}
			-\Delta u=f( u),\ \ u\ge 0,&\text{in}\ \ B_R,\\
             u=0,\ \ &\text{on}\  \ \partial B_R.
\end{cases}
	\end{align*}
In addition, suppose that $u\in C^1(V)$ with $V:=\{x\in B_R: 0<u(x)<\sup_{B_R} u\}$.
  The $u$ is locally symmetric (consider as functions defined on $\R^2$ with zero extension outside).
\end{proposition}

\begin{remark}
  Note that if $u$ is locally symmetric in $B_R$, then $u$ is radially symmetric and radially decreasing in annuli $A_k\, (k\in K)$, and flat elsewhere in $B_R$.
\end{remark}

The following result follows from Theorem 6.1 and Corollary 7.6 in \cite{Bro1}.
\begin{proposition}[\cite{Bro1}]\label{pro2-2}
  Let $0<r<R<+\infty$ and $f$ be of class $ C( [0, +\infty))$. Suppose $u\in H^1_0(B_R)\cap C(\overline{B_R})$ is a weak solution of the following problem
	\begin{align*}
		\begin{cases}
			-\Delta u=f( u),\ \ 0\le u \le 1&\text{in}\ \ B_R\backslash Q_r,\\
             u\equiv 1,\ \ &\text{in}\  \  Q_r.
\end{cases}
	\end{align*}
In addition, suppose that $u\in C^1(V)$ with $V:=\{x\in B_R: 0<u(x)<1\}$. Then $u$ is locally symmetric (after extending  by zero to $\R^2$).
\end{proposition}

\section{Proofs of Theorems \ref{th1} and \ref{th2}}\label{s3}
In this section, we give the proofs of Theorems \ref{th1} and \ref{th2}. The main idea is to consider the stream function which satisfies a semilinear equation, and then to use `local' symmetry properties for the non-negative solutions of the problems to conclude.
\subsection{Proof of Theorem \ref{th1}}
Let $u\in C^3(\overline{D}\backslash\{z\})\cap W_0^{1,\infty}(D)$ be the corresponding stream function of the steady flow $\mv$; see Proposition \ref{pro1}. Since $z\in D$ is the only stagnation point of $\mathbf{v}$ in $D$, the stream function $u$ has a unique critical point in $D$. We may assume, without loss of generality, that $u$ has a unique maximum point in $\overline{D}$ and this point is actually the stagnation point $z$ (after possibly changing $\mathbf{v}$ into $-\mathbf{v}$ and $u$ into $-u$).
 The uniqueness of the critical point of $u$ in $D$ implies that
\begin{equation*}
  0<u(x)<u(z)\ \ \ \text{for all}\ \ x\in D\backslash\{z\}.
\end{equation*}
By Proposition \ref{pro1}, we see that there is a continuous function $f: [0, u(z)) \to \mathbb{R}$ such that
  \begin{equation*}
   \Delta u+f(u)=0\ \ \ \text{in}\ \ D\backslash\{z\}.
  \end{equation*}
Taking into account the vorticity $\omega=\Delta u \in L^1(D)$, and using a standard cut-off argument, we check that $u$ is a weak solution of the equation $\Delta u+f(u)=0$ in $D$.
It remains to show that $u$ is a radially decreasing function with respect to the origin. Indeed, from Proposition \ref{pro2}, one knows that $u$ is locally symmetric, namely, it is  radially symmetric and radially decreasing in some annuli (probably infinitely many) and flat elsewhere. Recall that $u$ has a unique critical point in $D$. We conclude that the number of annuli can only be one at most, and hence $u$ is a radially decreasing function. The radial symmetry of stream function $u$ means that the flow $\mv$ is circular. The proof is thus complete.

\subsection{Proof of Theorem \ref{th2}}
Consider the corresponding stream function $u\in C^3(\overline{D})$ of the steady flow $\mv$. Up to normalization, we may assume, without loss of generality, that
\begin{equation*}
  u=1\  \text{on}\ \, C_a\ \ \ \text{and}\ \ \ u=0\ \, \text{on}\  C_b.
\end{equation*}
Since $\mathbf{v}$ has no stagnation point in $\Omega_{a, b}$, the stream function $u$ has no critical point in $\Omega_{a, b}$. It follows that
\begin{equation*}
  0<u(x)<1\ \ \ \text{for all}\ \ x\in \Omega_{a, b}.
\end{equation*}
By Proposition \ref{pro1-2}, there is a continuous function $f: [0, 1] \to \mathbb{R}$ such that
  \begin{equation*}
   \Delta u+f(u)=0\ \ \ \text{in}\ \ \Omega_{a, b}.
  \end{equation*}
It remains to show that $u$ is a radially decreasing function with respect to the origin. Set
\begin{equation*}
  \bar{u}(x)=\begin{cases}
               1, & \mbox{if }\ |x|\le a, \\
               u(x), & \mbox{if}\ a<|x|<b.
             \end{cases}
\end{equation*}
Then $\bar{u}\in  H^1_0(B_b)\cap C(\overline{B_b})$, where $B_b=\{x\in \mathbb{R}^2: |x|<b\}$. Moreover, $\bar{u}$ is a weak solution of the following problem
	\begin{align*}
		\begin{cases}
			-\Delta w=f(w),\ \ 0< w < 1&\text{in}\ \ \Omega_{a, b},\\
             w(x)\equiv 1,\ \ &\text{if}\ \ |x|\le a,\\
             w\in H^1_0(B_b).
\end{cases}
	\end{align*}
It follows from Proposition \ref{pro2-2} that $\bar u$ is locally symmetric, namely, it is  radially symmetric and radially decreasing in some annuli (probably infinitely many) and flat elsewhere. Recall that $u$ has no critical point in $\Omega_{a, b}$. We conclude that the number of annuli can only be one at most, and hence $u$ is a radially decreasing function. The proof is thus complete.

\section{Proofs of Theorems \ref{th3} and \ref{th4}}\label{s4}
In this section, we give the proofs of Theorems \ref{th3} and \ref{th4}. These theorems can be proved by the method analogous to that used in the previous section. In the previous section, the domains are a priori symmetric, so the `local' symmetry properties for the non-negative solutions of semi-linear elliptic equations can be applied directly. In our following consideration, none of specific symmetry on the domains is required. Our strategy is to extend the stream function appropriately into a larger symmetric domain, and then use the `local' symmetry results. The legitimacy of this extension is guaranteed by the enhanced boundary conditions.

\subsection{Proof of Theorem \ref{th3} }
 Consider the corresponding stream function $u\in C^3(\overline{D})$ of the steady flow $\mv$. Up to translation, we may assume that the unique critical point of $u$ in $D$ is the origin. Furthermore, by adding constants or multiplying $u$ by a constant number, we may assume, without loss of generality, that
\begin{equation*}
    u=0\ \, \text{on} \  \partial D,\ \ \ u(0)=1,\ \ \ \text{and}\ \ \ 0<u<1\ \, \text{in}\ D\backslash \{0\}.
\end{equation*}
 Proceeding as in the proof of Theorem \ref{th1} and using the continuity of $\Delta u$, we see that there exists a function $f\in C([0,1])\cap C^1\left((0,1)\right)$ such that
  \begin{equation*}
   \Delta u+f(u)=0\ \ \ \text{in}\ \ D.
  \end{equation*}
Let $R>0$ be such that $\overline{D}\subset B_R$. We now extend $u$ by zero to $B_R$, still denoted by $u$. Since $\nabla u=0$ on $\partial D$, it holds that $u\in C^1(\overline{B_R})$. Moreover, $u$ is a weak solution to the following problem:
	\begin{align*}
		\begin{cases}
			-\Delta u=g(u),\ \ u\ge 0,&\text{in}\ \ B_R,\\
             u=0,\ \ &\text{on}\  \ \partial B_R,
\end{cases}
	\end{align*}
where
\begin{equation*}
  g(t)=f(t)-f(0)\chi(t)\  \ \text{with}\ \ \chi(t):=\begin{cases}
                                         1, & \mbox{if }\  t=0, \\
                                         0, & \mbox{if}\ \ t>0.
                                       \end{cases}
\end{equation*}
By Proposition \ref{pro2}, we see that $u$ is locally symmetric. Recall that $\nabla u\not =0$ in $D\backslash\{0\}$ and $D$ is connected. We conclude that the union has only one term, that is, $D$ is a disk centered at the origin, and $u$ is radially symmetric and decreasing. The proof is thereby complete.

\subsection{Proof of Theorem \ref{th4} } The proof is similar to that of Theorem \ref{th3}. Consider the corresponding stream function $u\in C^3(\overline{D})$ of the steady flow $\mv$. By adding constants or multiplying $u$ by a constant number, we may assume, without loss of generality, that
\begin{equation*}
    u=1\ \, \text{on} \  \partial \Omega_1,\ \ \ u=0 \ \, \text{on} \  \partial \Omega_2    \ \ \ \text{and}\ \ \ 0<u<1\ \, \text{in}\ D.
\end{equation*}
There exists a function $f\in C([0,1])\cap C^1\left((0,1)\right)$ such that
  \begin{equation*}
   \Delta u+f(u)=0\ \ \ \text{in}\ \ D.
  \end{equation*}
  Let $R>0$ be such that $\overline{D}\subset B_R$. We now extend $u$ by $1$ inside $\Omega_1$ and by zero outside $\Omega_2$, still denoted by $u$. We check that $u\in C^1(\overline{B_R})$ is a weak solution to the following problem:
	\begin{align*}
		\begin{cases}
			-\Delta u=g(u),\ \ u\ge 0,&\text{in}\ \ B_R,\\
             u=0,\ \ &\text{on}\  \ \partial B_R,
\end{cases}
	\end{align*}
where
\begin{equation*}
  g(t)=f(t)-f(0)\chi_1(t)-f(1)\chi_2(t),\  \ \ \chi_1(t):=\begin{cases}
                                         1, & \mbox{if }\  t=0, \\
                                         0, & \mbox{if}\ \ t>0,
                                       \end{cases}\ \ \  \chi_2(t):=\begin{cases}
                                         0, & \mbox{if }\  0\le t<1, \\
                                         1, & \mbox{if}\ \ t\ge 1.
                                         \end{cases}
\end{equation*}
   It follows from Proposition \ref{pro2} that $u$ is locally symmetric. Recall that $\nabla u\not =0$ in $D$ and $D$ is connected. We conclude that the union has only one term, that is, $D$ is an annulus, and $u$ is radially symmetric and decreasing. The proof is thereby complete.

\begin{remark}
  We remark that partial cases of Theorems \ref{th3} and \ref{th4} can also be directly proved by using radial symmetry of $C^3$ solutions to overdetermined elliptic equations with non-Lipschitz nonlinearity; see Theorems 6.1 and 5.1 in \cite{Ruiz}. Here, for the sake of brevity, we have adopted a unified approach.
\end{remark}

\section{Proof of Theorem \ref{th6}}\label{s5}

In this section, we give the proof of Theorem \ref{th6}. Here the situation is a little more complicated. The previous method does not seem to work anymore, and we need to establish some new Serrin-type results.
\subsection{Proof of Theorem \ref{th6}}  Consider the corresponding stream function $u$ of the steady flow $\mv$. By adding constants or multiplying $u$ by a constant number, we may assume, without loss of generality, that
\begin{equation*}
    u=1\ \, \text{on} \  \partial \Omega_1,\ \ \ u=0 \ \, \text{on} \  \partial \Omega_2    \ \ \ \text{and}\ \ \ 0<u<1\ \, \text{in}\ D.
\end{equation*}
By Proposition \ref{pro1-2}, there exists a function $f\in C([0,1])\cap C^1\left((0,1)\right)$ such that
  \begin{equation*}
   \Delta u+f(u)=0\ \ \ \text{in}\ \ D.
  \end{equation*}
By assumption, we have that $|\nabla u|=c_1$ on $\partial \Omega_1$ and $|\nabla u|=c_2$ on $\partial \Omega_2$ for some non-negative numbers $c_1$ and $c_2$. In view of Theorem \ref{th4} proved earlier and Theorem 1.13 of \cite{HN}, it remains to consider the case when only one of $c_1$ and $c_2$ is equal to zero. Let us consider the two cases separately.

Assume first that $c_1=0$ and $c_2\not=0$. If $f(1)\not=0$, then the conclusion immediately follows from Proposition \ref{pro3} below. We now assume that $f(1)=0$. We extend $u$ by $1$ inside $\Omega_1$, still denoted by $u$. We check that $u\in C^1(\overline{\Omega_2})$ is a weak solution of the following problem:
  \begin{align*}
		\begin{cases}
			-\Delta u=f(u),\ \ u> 0,&\text{in}\ \ \Omega_2,\\
             u=0,\ \ |\nabla u|=c_2, &\text{on}\ \, \partial \Omega_2.
\end{cases}
	\end{align*}
By Proposition \ref{pro4} below, we see that $\Omega_2=B_R(x)$ for some $R>0$ and $x\in \R^2$. Moreover, $u$ is locally symmetric in $B_R(x)$. Recall that $\nabla u\not =0$ in $D$ and $D$ is connected. We conclude that the union has only one term, that is, $D$ is an annulus, and $u$ is radially symmetric and decreasing.

Assume now that $c_1\not=0$ and $c_2=0$. If $f(0)\not=0$, then the conclusion immediately follows from Proposition \ref{pro3} below. It remains to consider the case when $f(0)=0$.
Note that $u\in C^3(\overline{D})$ is a solution of the following problem:
  \begin{align*}
		\begin{cases}
			-\Delta u=f(u),\ \ 0< u<1,&\text{in}\ \ D,\\
             u=1,\ \ |\nabla u|=c_1>0, &\text{on}\ \, \partial \Omega_1,\\
             u=0,\ \ |\nabla u|=0, &\text{on}\ \, \partial \Omega_2.
\end{cases}
	\end{align*}
By Proposition \ref{pro5} below, we see that $u$ is locally symmetric. Recall that $\nabla u\not =0$ in $D=\{x\in \R^2: 0<u(x)<1\}$ and $D$ is connected. We conclude that the union has only one term, that is, $D$ is an annulus, and $u$ is radially symmetric and decreasing. The proof is thereby complete.

\subsection{Serrin-type results}
Below we state and prove some Serrin-type results on the symmetries for overdetermined elliptic problems, which have been used in the proof of Theorem \ref{th6}. We refer the reader to \cite{Bro2, Far, Nit, Rei, Se, Si} and the references therein for more extensive discussions on the Serrin-type overdetermined problems.

We first recall two useful symmetry results due to Ruiz \cite{Ruiz} and Brock \cite{Bro2}.
\begin{proposition}[\cite{Ruiz}, Theorem 5.1]\label{pro3}
  Let $\Omega_1$ and $\Omega_2$ be two $C^2$ non-empty simply connected bounded domains of $\mathbb{R}^2$ such that $\overline{\Omega_1}\subset \Omega_2$, and denote
  \begin{equation*}
    D=\Omega_2\backslash \overline{\Omega_1}.
  \end{equation*}
  Let $f\in C([0,1])\cap C^1(0,1)$ and $u\in C^3(\overline{D})$ be a solution of the overdetermined problem
  	\begin{align*}
		\begin{cases}
			-\Delta u=f(u),\ \ u\in (0,1),&\text{in}\ \ D,\\
             u=1,\ \ |\nabla u|=c_1,&\text{on}\  \ \partial \Omega_1,\\
             u=0,\ \ |\nabla u|=c_2,&\text{on}\  \ \partial \Omega_2,\\
\end{cases}
	\end{align*}
where $c_1$ and $c_2$ are constant numbers. If $c_i=0$ we also assume that $f(2-i)\not=0,\, i=1,2$. Then $u$ is a radially symmetric function with respect to a point $x\in \R^2$ and $\Omega_i=B_{R_i}(x)$, $R_2>R_1>0$. Moreover, $u$ is strictly radially decreasing.
\end{proposition}

\begin{proposition}[\cite{Bro2}, Theorem 1]\label{pro4}
  Let $D$ be a bounded planar domain with $C^2$ boundary. Let $f\in C([0, +\infty))$ and $u\in C^1(\overline{D})$ be a weak solution of the following problem
  \begin{align*}
		\begin{cases}
			-\Delta u=f(u),\ \ u> 0,&\text{in}\ \ D,\\
             u=0,\ \ |\nabla u|=c, &\text{on}\  \ \partial D,
\end{cases}
	\end{align*}
where $c>0$ is a constant number. Then $D=B_R(x)$ for some $R>0$ and $x\in \R^2$. Moreover, $u$ is locally symmetric (after extending  by zero to $\R^2$).
\end{proposition}

\begin{proposition}\label{pro5}
  Let $\Omega_1$ and $\Omega_2$ be two $C^2$ non-empty simply connected bounded domains of $\mathbb{R}^2$ such that $\overline{\Omega_1}\subset \Omega_2$, and denote
  \begin{equation*}
    D=\Omega_2\backslash \overline{\Omega_1}.
  \end{equation*}
 Let $f\in C([0, +\infty))$ satisfy $f(0)=0$ and $u\in C^1(\overline{D})$ be a weak solution of the following problem:
  \begin{align*}
		\begin{cases}
			-\Delta u=f(u),\ \ 0< u<1,&\text{in}\ \ D,\\
             u=1,\ \ |\nabla u|=c, &\text{on}\ \, \partial \Omega_1,\\
             u=0,\ \ |\nabla u|=0, &\text{on}\ \, \partial \Omega_2.
\end{cases}
	\end{align*}
where $c>0$ is a constant number. Then $\Omega_i=B_{R_i}(x)$, $i=1,2$, $R_2>R_1>0$. Moreover, if we extend $u$ by $1$ inside $\Omega_1$ and by $0$ outside $\Omega_2$, denoted by $\bar{u}$, then $\bar{u}$ is locally symmetric.
\end{proposition}
The proof of Proposition \ref{pro5} is similar to that of Proposition \ref{pro4} in \cite{Bro2}, which is based on the  method of continuous Steiner symmetrization (CStS). For the convenience of the reader we first recall the definition of the continuous Steiner symmetrization (CStS); see \cite{Bro0, Bro1, Bro3, Bro2, Soly}. We shall draw on the presentation of \cite{Bro2}.

We start with some notation. Let $N\ge 2$ be an integer. For points $x\in \R^N$, we write $x=(x_1, x')$, where $x'=(x_2, \cdots, x_N)$. Let $\mathcal{L}^k$ denote $k$-dimensional Lebesgue measure. By $\mathcal{M}(\R^N)$ we denote the family of Lebesgue measurable sets in $\R^N$ with finite measure. For a function $u:\R^N\to \R$, let $\{u>a\}$ and $\{b\ge u>a\}$ denote the sets $\left\{x\in \R^N: u(x)>a \right\}$ and $\left\{x\in \R^N: b\ge u(x)>a \right\}$, respectively, ($a, b\in \R$, $a<b$). Let $\mathcal{S}(\R^N)$ denote the class of real, nonnegative measurable functions $u$ satisfying
\begin{equation*}
  \LL^N(\{u>c\})<+\infty,\ \ \forall\, c>\inf u.
\end{equation*}

The following is the definition of the well-known Steiner symmetrization; see e.g. \cite{Bro3, Kaw, Lieb}.
\begin{definition}[Steiner symmetrization]
\ \ \
  \begin{itemize}
    \item [(i)]For any set $M\in \M(\R)$ let
    \begin{equation*}
      M^*:=\left(-\frac{1}{2}\LL^1(M),\  \frac{1}{2}\LL^1(M)\right).
    \end{equation*}
    \item [(ii)]Let $M\in \M(\R^N)$. For every $x'\in \R^{N-1}$ let
    \begin{equation*}
      M(x'):=\left\{x_1\in \R: (x_1, x')\in M \right\}.
    \end{equation*}
    The set
    \begin{equation*}
      M^*:=\left\{ x=(x_1, x'): x_1\in \left(M(x') \right)^*, x'\in \R^{N-1}\right\}.
    \end{equation*}
    is called the Steiner symmetrization of $M$ (with respect to $x_1$).
    \item [(iii)]If $u\in \s$, then the function
    \begin{equation*}
      u^*(x):=\begin{cases}
                \sup \left\{c>\inf u: x\in \left\{u>c \right\}^*\right\}, & \mbox{if }\  x\in \bigcup_{c>\inf u} \left\{u>c \right\}^*, \\
                \inf u, & \mbox{if }\  x\not\in \bigcup_{c>\inf u} \left\{u>c \right\}^*,
              \end{cases}
    \end{equation*}
   is called the Steiner symmetrization of $u$ (with respect to $x_1$).
  \end{itemize}
\end{definition}

\begin{definition}[Continuous symmetrization of sets in $\M(\R)$]
  A family of set transformations
  \begin{equation*}
    \T_t:\  \M(\R)\to \M(\R),\ \ \  0\le t\le +\infty,
  \end{equation*}
is called a continuous symmetrization on $\R$ if it satisfies the following properties: ($M, E\in \M(\R)$, $0\le s, t\le +\infty$)
\begin{itemize}
  \item [(i)]Equimeasurability property:\, $\LL^1(\T_t(M))=\LL^1(M)$,

    \smallskip
  \item [(ii)]Monotonicity property:\, If $M\subset E$, then $\T_t(M)\subset \T_t(E)$,

    \smallskip
  \item [(iii)]Semigroup property:\, $\T_t(\T_s(M))=\T_{s+t}(M)$,

    \smallskip
  \item [(iv)]Interval property:\, If $M$ is an interval $[x-R,\ x+R]$, ($x\in \R$, $R>0$), then $\T_t(M):=[xe^{-t}-R,\ xe^{-t}+R]$,

    \smallskip
  \item [(v)]Open/compact set property: If $M$ is open/compact, then $\T_t(M)$ is open/compact.
\end{itemize}
\end{definition}

The existence and uniqueness of the family $\T_t$, $0\le t \le +\infty$, can be found in \cite{Bro1}, Theorem 2.1.

\begin{definition}[Continuous Steiner symmetrization (CStS)]
  \ \ \
  \begin{itemize}
    \item [(i)]Let $M\in \M(\R^N)$. The family of sets
    \begin{equation*}
      \T_t(M):=\left\{x=(x_1, x'): x_1\in \T_t(M(x')), x'\in \R^{N-1} \right\},\ \ \ 0\le t\le +\infty,
    \end{equation*}
    is called the continuous Steiner symmetrization (CStS) of $M$ (with respect to $x_1$).
    \item [(ii)]Let $u\in \s$. The family of functions $\T_t(u)$, $0\le t \le +\infty$, defined by
    \begin{equation*}
      \T_t(u)(x):=\begin{cases}
                \sup \left\{c>\inf u: x\in \T_t\left(\left\{u>c \right\}\right)\right\}, & \mbox{if }\ x\in \bigcup_{c>\inf u} \T_t\left(\left\{u>c \right\}\right), \\
                \inf u, & \mbox{if }\ x\not\in \bigcup_{c>\inf u} \T_t\left(\left\{u>c \right\}\right),
              \end{cases}
    \end{equation*}
    is called CStS of $u$ (with respect to $x_1$).
  \end{itemize}
\end{definition}
For convenience, we will henceforth simply write $M^t$ and $u^t$ for the sets $\T_t(M)$, respectively for the functions $\T_t(u)$, ($t\in [0,+\infty]$). Below we summarize basic properties of CStS, which have been proved by Brock in \cite{Bro0, Bro1}.

\begin{proposition}\label{pro6}
  Let $M\in \M(\R^N)$, $u,v\in \s,\, t\in [0,+\infty]$. Then
  \begin{itemize}
    \item [(1)]Equimeasurability:
    \begin{equation*}
      \LL^N(M)=\LL^N(M^t)\ \ \ \text{and}\ \ \ \left\{u^t>c \right\}=\left\{u>c \right\}^t,\  \forall\, c>\inf u.
    \end{equation*}

    \smallskip
    \item [(2)]Monotonicity: If $u\le v$, then $u^t\le v^t$.

    \smallskip
    \item [(3)]Commutativity: If $\phi: [0, +\infty)\to [0, +\infty)$ is bounded and nondecreasing, then
    \begin{equation*}
      \phi(u^t)=[\phi(u)]^t.
    \end{equation*}

    \smallskip
    \item [(4)]Homotopy:
    \begin{equation*}
      M^0=M,\ \ \ , u^0=u,\ \ \ M^\infty =M^*,\ \ \ u^\infty=u^*.
    \end{equation*}
Furthermore, from the construction of the CStS it follows that, if $M=M^*$ or $u=u^*$, then $M^t=M$, respectively, $u=u^t$ for all $t\in [0, +\infty]$.

    \smallskip
        \item [(5)]Cavalieri's pinciple: If $F$ is continuous and if $F(u)\in L^1(\R^N)$ then
        \begin{equation*}
          \int_{\R^N} F(u)\d x= \int_{\R^N} F(u^t)\d x.
        \end{equation*}

    \smallskip
    \item [(6)]Continuity in $L^p$: If $t_n\to t $ as $n\to +\infty$ and $u\in L^p(\R^N)$ for some $p\in [1, +\infty)$, then
    \begin{equation*}
      \lim_{n\to +\infty}\|u^{t_n}-u^t\|_p=0.
    \end{equation*}

    \smallskip
        \item [(7)]Nonexpansivity in $L^p$: If $u, v\in L^p(\R^N)$ for some $p\in [1, +\infty)$, then
        \begin{equation*}
          \|u^{t}-v^t\|_p\le \|u-v\|_p.
        \end{equation*}

          \smallskip
        \item [(8)]Hardy-Littlewood inequality: If $u, v\in L^2(\R^N)$ then
        \begin{equation*}
           \int_{\R^N} u^t v^t\d x\ge \int_{\R^N} u v\d x.
        \end{equation*}

         \smallskip
        \item [(9)]If $u$ is Lipschitz continuous with Lipschitz constant $L$, then $u^t$ is Lipschitz continuous, too, with Lipschitz constant less than or equal to $L$.

         \smallskip
        \item [(10)] If $\text{supp}\, u\subset B_R$ for some $R>0$, then we also have $\text{supp}\, u^t\subset B_R$. If, in addition, $u$ is Lipschitz continuous with Lipschitz constant $L$, then we have
            \begin{equation*}
              |u^t(x)-u(x)|\le LR\, t,\ \ \ \forall\, x\in B_R.
            \end{equation*}
           Furthermore, there holds
           \begin{equation}\label{6-1}
             \int_{B_R}G(|\nabla u^t|)\d x \le  \int_{B_R}G(|\nabla u|)\d x,
           \end{equation}
           for every convex function $G: [0, +\infty) \to [0, +\infty)$ with $G(0)=0$.
  \end{itemize}
\end{proposition}

\begin{definition}[Local symmetry]\label{def5-5}
  Let $u\in C^1(\R^2)$ be a nonnegative function with compact support. If $y=(y_1, y')\in \R^2$ with
  \begin{equation*}
    0<u(y)<\sup u,\ \ \ \frac{\partial u}{\partial x_1}(y)>0,
  \end{equation*}
  and $\tilde{ y}$ is the (unique) point satisfying
  \begin{equation*}
    \tilde{y}=(\tilde{y_1}, y'),\ \ \ \tilde{y_1}>y_1,\ \ \ u(y)=u(\tilde{y})< u(s, y'),\ \ \forall\, s\in (y_1,\tilde{y_1}),
  \end{equation*}
  then
  \begin{equation*}
  \frac{\partial u}{\partial x_2}(y) =\frac{\partial u}{\partial x_2}(\tilde{y}),\ \ \   \frac{\partial u}{\partial x_1}(y)  =-\frac{\partial u}{\partial x_1}(\tilde{y}).
  \end{equation*}
  Then $u$ is called locally symmetric in the direction $x_1$.
\end{definition}
\begin{remark}\label{re100}
  As remarked by Brock (see Remark 6.1 in \cite{Bro1}), if $u$ is locally symmetric in direction $x_1$, then the following disjoint decomposition holds
  \begin{equation}\label{eq100}
    \left\{0<u<\sup u\right\}=\bigcup_{k=1}^{m}\left(U_1^k\cup U_2^k\right)\bigcup S.
  \end{equation}
  Here $U_1^k$ is some maximal connected component of $\{0<u<\sup u\}\cap \{\partial_{x_1}u>0\}$, $U_2^k$ is its reflection about some line $\{x_1=d_k\},\, d_k\in \R$, and we have
  \begin{equation*}
    \partial_{x_1}u=0\ \ \ \text{in}\ S,
  \end{equation*}
  and for every $(y_1, y')\in U_1^k$,
  \begin{equation*}
    u(y_1, y')=u(2d_k-y_1, y')<u(2d_k-y_1, s),\ \ \ \forall\, s\in (y_1, 2d_k-y_1),\ k=1, \cdots, m.
  \end{equation*}
   There can be a countable number of $U_1^k$'s, i.e., $m=+\infty$.
\end{remark}

We also need the following symmetry criterion due to Brock \cite{Bro1}.
\begin{lemma}[\cite{Bro1}, Theorem 6.2]\label{le1}
  Let $R>0$. Let $u\in H^{1}_0(B_R) \cap C(\overline{B_R})$ be such that $u>0$ in $B_R$ and $u$ is continuously differentiable on $\left\{x\in B_R: 0<u(x)<\sup u\right\}$. Suppose that
  \begin{equation*}
    \lim_{t\to 0}\frac{1}{t}\left(\int_{B_R} |\nabla u|^2\d x-\int_{B_R} |\nabla u^t|^2\d x\right)=0.
  \end{equation*}
  Then $u$ is locally symmetric in direction $x_1$.
\end{lemma}

It can be proved that functions which are locally symmetric in \emph{every} direction are locally (radially) symmetric in the sense of Definition \ref{de1}.
\begin{lemma}[\cite{Bro1}, Theorem 6.1]\label{le2}
  Let $u$ be as in Lemma \ref{le1}. Suppose that for arbitrary rotations $x\mapsto y=(y_1, y')$ of the coordinate system, $u$ is locally symmetric in direction $y_1$. Then $u$ is locally symmetric (after extending by zero to $\R^2$). Moreover, the super-level sets $\{u>t\}$ $(t\ge 0)$ are countable unions of mutually disjoint balls.
\end{lemma}

We are now in a position to prove Proposition \ref{pro5}.

\smallskip
\noindent \emph{Proof of Proposition \ref{pro5}}\,:\, We choose a number $R>0$ such that $\overline{\Omega_2}\subset B_R$. We extend $u$ by $1$ inside $\Omega_1$ and by $0$ outside $\Omega_2$, still denoted by $u$. We denote by $L$ the Lipschitz constant of $u$. Let $u^t$, ($0\le t\le +\infty$), denote the CStS of $u$ with respect to $x_1$. By Proposition \ref{pro6}, $u^t$ is Lipschitz continuous with Lipschitz constant less than or equal to $L$, $\text{supp}\, u^t\subset B_R$, and
\begin{equation}\label{5-1}
  |u^t(x)-u(x)|\le LR\,t,\ \ \ \forall\,x\in \R^2,\ t\in[0, +\infty].
\end{equation}
Set $d(x):=\inf\left\{|x-z|: z\in \partial \Omega_1 \right\}$. Clearly, it holds
\begin{equation}\label{5-2}
  \lim_{s\to 1}\sup \left\{d(x): s\le u(x)<1  \right\}=0.
\end{equation}
Let $\kappa:=2RL$. We claim that for every $t\in [0, +\infty)$, $u^t+kt\ge 1$ on $\overline{ \Omega_1}$. Indeed, by \eqref{5-2}, we have for every $x\in \overline{ \Omega_1}$,
\begin{equation*}
  u^t(x)\ge u(x)-|u^t(x)-u(x)|\ge 1-RLt\ge 1-\kappa t.
\end{equation*}
For convenience, we set
\begin{equation*}
  \begin{split}
      M_1(t) & :=\left\{x\in D: 1-\kappa t\le u(x) <1 \right\} \\
      M_2(t) & :=\left\{x\in D: 1-\kappa t\le u^t(x) <1 \right\},\ \ \ (t\in(0, +\infty)).
  \end{split}
\end{equation*}
By Proposition \ref{pro6}, we have
\begin{equation}\label{5-3}
  \LL^2(M_1(t))\ge \LL^2(M_2(t)),\ \ \ \forall\,t\in(0, +\infty).
\end{equation}
In view of \eqref{5-1} and \eqref{5-2}, it is easy to see that
\begin{equation}\label{5-4}
   \lim_{t\to 0}\sup \left\{d(x): x\in M_1(t)\cup M_2(t)  \right\}=0.
\end{equation}
We now show that there exists a constant $C_0>0$ such that
\begin{equation}\label{5-5}
  \LL^2(M_1(t))\le C_0\, t,\ \ \ (0<t<+\infty).
\end{equation}
By \eqref{5-4} and the fact that $|\nabla u|=c$ on $\partial \Omega_1$, there exists a small $t_0\in(0, 1/\kappa)$ such that
\begin{equation}\label{5-6}
  |\nabla u(x)|\ge c/2\ \ \ \text{if}\ \, x\in M_1(t_0).
\end{equation}
By the Implicit Function Theorem, we have that for every $s\in [1-\kappa t_0, 1)$, the level set $\left\{u=s\right\}$ is a $C^1$ Jordan curve in $D$. Recall that
\begin{equation}\label{5-7}
  -\Delta u =f(u)\ \ \ \text{in}\ \, D.
\end{equation}
Integrating \eqref{5-7} over the open set $\left\{x\in D: s<u(x)<1 \right\}$, we get
\begin{equation*}
  -\int_{\{u=1\}}|\nabla u|\d \mathcal{H}_1(x)+\int_{\{u=s\}} |\nabla u|\d \mathcal{H}_1(x)=\int_{\{s<u<1\}} f(u)\d x.
\end{equation*}
By \eqref{5-6} this implies
\begin{equation*}
  \int_{\{u=s\}} \d \mathcal{H}_1(x)\le C_1
\end{equation*}
for some positive number $C_1$.  Using the co-area formula (see \cite{Evans}), we obtain
\begin{equation*}
\begin{split}
    \LL^2(M_1(t)) & \le \frac{2}{c} \int_{\{1-\kappa t\le u<1\}} |\nabla u|\d x =\frac{2}{c} \int_{1-\kappa t}^{1}\left( \int_{\{u=s\}} \d \mathcal{H}_1(x) \right) \d s \\
     & \le \frac{2C_1}{c}\kappa t,\ \  (t\in(0, t_0]),
\end{split}
\end{equation*}
and the assertion follows.

Set $w^t=\min \{u^t, 1-\kappa t\}$ and $w^0=\min \{u, 1-\kappa t\}$ for $t\in (0, 1/\kappa)$. We check that $w^t\equiv 1-\kappa t$ in $\overline{\Omega_1}$ and $w^t=0$ on $\partial B_R$. Since
\begin{equation*}
    -\Delta u =f(u)\ \ \ \text{in}\ \, B_R\backslash \overline{\Omega_1}=: U,
\end{equation*}
we have
\begin{equation}\label{5-15}
  0=\int_U \nabla u\cdot \nabla (w^t-w^0)\d x-\int_{U}f(u)(w^t-w^0)\d x=:I_1(t)-I_2(t),\ \ \ t\in [ 0, 1/\kappa).
\end{equation}
By the symbol $o(t)$ we denote any function satisfying $\lim_{t\to 0} o(t)/t=0$.
We first claim that
\begin{equation}\label{5-8}
  I_2(t)\ge o(t).
\end{equation}
To show \eqref{5-8}, we split $I_2(t)$ into the following two parts:
\begin{equation*}
  I_2(t)=\int_{M_1(t)}[f(u)-f(1-\kappa t)](w^t-w^0)\d x+\int_{U}f(w^0)(w^t-w^0)\d x=:I_{21}(t)+I_{22}(t).
\end{equation*}
By \eqref{5-1} we have
\begin{equation}\label{5-9}
  |w^t(x)-w^0(x)|\le LR\,t,\ \ \ \forall\, x\in \R^2.
\end{equation}
Set $\eta:=\sup_{t\in [0,1]}f(t)$. Using \eqref{5-5}, we get
\begin{equation}\label{5-11}
  |I_{21}(t)|\le 2\eta LRC_0t^2.
\end{equation}
Set $F(t):=\int_{0}^{t}f(s)\d s$. By Proposition \ref{pro6}, we have
\begin{equation*}
  w^t=\min \{u^t, 1-\kappa t\}=\left(\min \{u, 1-\kappa t\}\right)^t=(w^0)^t,
\end{equation*}
and
\begin{equation*}
  \int_{B_R} F(w^0)\d x=\int_{B_R}F(w^t)\d x.
\end{equation*}
Since $w^t\equiv 1-\kappa t$ in $\overline{\Omega_1}$, we infer that
\begin{equation}\label{5-10}
0=\int_U\left[ F(w^t)-F(w^0)\right]\d x=\int_U \int_{0}^{1}f(w^0+\theta(w^t-w^0))\d \theta\,(w^t-w^0)\d x.
\end{equation}
Set
\begin{equation*}
  \lambda(t):=\sup \left\{|f(w^0(x)+\theta(w^t(x)-w^0(x)))-f(w^0(x))|: x\in B_R,\  \theta\in [0,1] \right\}.
\end{equation*}
 In view of \eqref{5-9}, one can verify that $\lambda(t)=o(1)$ as $t\to 0$. Hence we have
\begin{equation}\label{5-8-1}
  I_{22}(t)\ge \int_U \int_{0}^{1}\left[f(w^0)-f(w^0+\theta(w^t-w^0)) \right]\d \theta(w^t-w^0)\d x\ge -\lambda(t)\cdot LRt\cdot \LL^2(U).
\end{equation}
Now \eqref{5-8} is just \eqref{5-11} and \eqref{5-8-1} combined.

Next we estimate $I_1(t)$. Notice that
\begin{equation}\label{5-12}
\begin{split}
   I_1(t) & =\int_U \nabla u\cdot \nabla (w^t-w^0)\d x=\int_{B_R\backslash M_2(t)} \nabla u\cdot \nabla w^t\d x-\int_{B_R\backslash M_1(t)} \nabla u\cdot \nabla w^0\d x \\
     & \le \frac{1}{2} \int_{B_R\backslash M_2(t)}|\nabla w^t|^2\d x+\frac{1}{2} \int_{B_R\backslash M_2(t)}|\nabla u|^2\d x-\int_{B_R\backslash M_1(t)} |\nabla u|^2\d x \\
     &= \frac{1}{2} \int_{B_R}|\nabla w^t|^2\d x- \frac{1}{2} \int_{B_R}|\nabla w^0|^2\d x+ \frac{1}{2} \int_{M_1(t)}|\nabla u|^2\d x- \frac{1}{2} \int_{M_2(t)}|\nabla u|^2\d x.
\end{split}
\end{equation}
Using \eqref{5-3}, \eqref{5-4} and the fact that $|\nabla u|=c$ on $\partial \Omega_1$, we obtain
\begin{equation}\label{5-13}
   \frac{1}{2} \int_{M_1(t)}|\nabla u|^2\d x- \frac{1}{2} \int_{M_2(t)}|\nabla u|^2\d x  \le \frac{1}{2} \int_{M_1(t)}(|\nabla u|^2-c^2)\d x- \frac{1}{2} \int_{M_2(t)}(|\nabla u|^2-c^2)\d x=o(t).
\end{equation}
Combining \eqref{5-15}, \eqref{5-8}, \eqref{5-12} and \eqref{5-13}, we see that
\begin{equation}\label{5-16}
  \int_{B_R}|\nabla w^t|^2\d x-  \int_{B_R}|\nabla w^0|^2\d x\ge o(t).
\end{equation}
Now fix $\varepsilon\in (0, 1)$. For functions $v$ we write $v_+:=\max \{v, 0\}$. For each $t\in [0, \varepsilon/\kappa)$, we set
\begin{equation*}
   h^0  =\min\left\{[u-(1-\varepsilon)]_+, 1-\kappa t  \right\},\ \ \  h^t=\min\left\{[u^t-(1-\varepsilon)]_+, 1-\kappa t  \right\}\ \ \ \text{and}\ \ \ g^t=\min\{u^t, 1-\varepsilon\}.
\end{equation*}
It follows from \eqref{6-1} that
\begin{equation*}
  \int_{\{1-\varepsilon<u^t<1-\kappa t\}}|\nabla u^t|^2\d x=\int_{B_R}|\nabla h^t|^2\d x \le \int_{B_R}|\nabla h^0|^2\d x=\int_{\{1-\varepsilon<u<1-\kappa t\}}|\nabla u|^2\d x.
\end{equation*}
On combining this with \eqref{6-1} and \eqref{5-16}, we deduce that
\begin{equation*}
  \begin{split}
     0&\ge \int_{B_R}|\nabla g^t|^2\d x-  \int_{B_R}|\nabla g^0|^2\d x   \\
       & =\left( \int_{B_R}|\nabla w^t|^2\d x-  \int_{B_R}|\nabla w^0|^2\d x \right)-\left(\int_{B_R}|\nabla h^t|^2\d x- \int_{B_R}|\nabla h^0|^2\d x \right)\\
       &\ge  \int_{B_R}|\nabla w^t|^2\d x-  \int_{B_R}|\nabla w^0|^2\d x \ge o(t),
  \end{split}
\end{equation*}
which implies
\begin{equation}\label{6-2}
  \lim_{t\to 0}\frac{1}{t}\left(\int_{B_R}|\nabla g^t|^2\d x-  \int_{B_R}|\nabla g^0|^2\d x\right)=0.
\end{equation}
In view of Lemma \ref{le1}, $g^0=\min\{u, 1-\varepsilon\}$ is locally symmetric in direction $x_1$. Note that the same estimate \eqref{6-2} can be established for CStS in arbitrary directions. It follows that $\min\{u, 1-\varepsilon\}$ is locally symmetric in every direction, and hence, by Lemma \ref{le2}, $\min\{u, 1-\varepsilon\}$ is locally symmetric and the super-level sets $\{u>t\}$ $(t\ge 0)$ are countable unions of mutually disjoint disks. Since $\varepsilon$ is arbitrary, $u$ must also be locally symmetric, and $D$ is an annulus. The proof is thus complete.

\section{Proof of Theorem \ref{th11}}\label{s6}
In this section, we give the proof of Theorem \ref{th11}. The proof is somewhat more intricate and necessitates leveraging the earlier results. Furthermore, relying merely on the continuous Steiner symmetrization techniques is insufficient to achieve the desired conclusion; hence, it is necessary to complement it with other methods.

We first recall a single-direction version of the Proposition \ref{pro2}. It is actually a special case of Theorem 7.2 in \cite{Bro1}, but it suffices for our purposes here. Recall the definition of local symmetry in Definition \ref{def5-5}.

\begin{proposition}\label{pro100}
  Let $D$ be a $C^2$ non-empty simply connected bounded domain of $\mathbb{R}^2$, such that it is Steiner symmetric with respect to the direction $\mathbf{e}_1=(1, 0)$ and the symmetry line is given by $\Gamma_0=\left\{x\in \R^2: x_1=0 \right\}$. Let $f\in L^\infty([0, +\infty))$ be such that
\begin{equation*}
  f=f_1+f_2+f_3,
\end{equation*}
where $f_1$ is a continuous function, $f_2$ is a non-increasing function, and $f_3$ is a non-decreasing function. Let $u\in H^1_0(D)\cap C(\overline{D})$ be a weak solution of the following problem
	\begin{align*}
		\begin{cases}
			-\Delta u=f( u),\ \ u\ge 0,&\text{in}\ \ D,\\
             u=0,\ \ &\text{on}\  \ \partial D.
\end{cases}
	\end{align*}
In addition, suppose that $u\in C^1(V)$ with $V:=\{x\in D: 0<u(x)<\sup_{D} u\}$.
  The $u$ is locally symmetric in direction $x_1$ (after extending  by zero to $\R^2$).
\end{proposition}

\subsection{Proof of Theorem \ref{th11}}
Firstly, if the boundary of the domain consists entirely of stagnation points, i.e., $\mv=0$ on $\partial D$, then the conclusion clearly follows from Theorem \ref{th3}. In this case, we can even assert that the domain $D$ is a disk and $\mv$ is a circular flow. Therefore, it remains to consider the case when $\{x\in \partial D: |\mathbf{v}(x)|=0\}\subsetneq \partial D$. Let $u\in C^3(\overline{D})$ be the corresponding stream function of $\mv$; see Proposition \ref{pro1}. We may also assume, without loss of generality, that
\begin{equation*}
    u=0\ \, \text{on} \  \partial D,\ \ \ u(z)=1,\ \ \ \text{and}\ \ \ 0<u<1\ \, \text{in}\ D\backslash \{z\}.
\end{equation*}
By Proposition \ref{pro1}, and using the continuity of $\Delta u$, we see that there exists a function $f\in C([0,1])\cap C^1\left([0,1)\right)$ such that
  \begin{equation}\label{199}
   \Delta u+f(u)=0\ \ \ \text{in}\ \ \overline{D}.
  \end{equation}
According to Proposition \ref{pro100}, we initially infer that $u$ is locally symmetric in direction $x_1$. In view of Remark \ref{re100}, the following disjoint decomposition holds
  \begin{equation}\label{eq102}
    D\backslash\{z\}=\bigcup_{k=1}^{m}\left(U_1^k\cup U_2^k\right)\bigcup S.
  \end{equation}
  Here $U_1^k$ is some maximal connected component of $\left\{x\in D\backslash\{z\}: \partial_{x_1}u(x)>0 \right\}$, $U_2^k$ is its reflection about some line $\{x_1=d_k\},\, d_k\in \R$, and we have $ \partial_{x_1}u=0$ in $S$.

 Our goal now is to show that $\partial_{x_1} u>0$ in $\{x\in D: x_1<0\}$, and hence $m=1$ in \eqref{eq102}. This further implies that $\partial_{x_1}u\not=0$ in $D\backslash \Gamma_0$, and therefore $z\in \Gamma_0$.  We will use the moving plane method to establish the desired result (see \cite{HN,Rei} for similar arguments). For $\lambda\in \R$, we denote
\begin{equation*}
  \Gamma_\lambda=\{x\in \R^2: x_1=\lambda\},\ \ \ \mathcal{H}_\lambda=\{x\in D: x_1<\lambda\},
\end{equation*}
and, for $x\in \R^2$,
\begin{equation*}
  \mathcal{R}_\lambda(x)=x^\lambda=(x_1-2\lambda, x_2).
\end{equation*}
In other words, $\mathcal{R}_\lambda$ is the orthogonal reflection with respect to the line $\Gamma_\lambda$.

    Assume now that $z=(z_1, z_2)$ with $z_1\ge 0$. If $z_1=0$, the proof is finished. Thus, it suffices to eliminate the remaining possibilities. Let us assume, for a contradiction, that $z_1>0$. Let $\varepsilon\in (0, 1)$ be a small number such that $\Im_\varepsilon=\{x\in D: u(x)=1-\varepsilon\} \subset \{x\in D: x_1>0\}$. Recall that the streamline $\Im_\varepsilon$ is a $C^2$ Jordan curve surrounding $z$ in $D$ and $\text{dist}(\Im_\varepsilon, z)\to 0$ as $\varepsilon \to 0^+$; see \cite{HN}. Set $U_\varepsilon=\{x\in D: u(x)>1-\varepsilon\}$. Let $V_\varepsilon=\{x\in D: 0<u(x)<1-\varepsilon\}$ be the doubly connected domain located between $\partial D$ and $\Im_\varepsilon$. Notice that
 \begin{equation*}
   \begin{cases}
      \Delta u +f(u)=0\ \ \ \mbox{in}\ V_\varepsilon, &  \\
     \ 0<u  <1\ \ \ \mbox{in}\ V_\varepsilon, &  \\
     u=0\ \text{on}\ \partial D,\ \ \ u   =1\ \text{on}\ \Im'_\varepsilon, &
   \end{cases}
 \end{equation*}
 where $f\in C^1([0, 1-\varepsilon])$. Denote $\bar{\lambda}=\min_{x\in D} x_1$. For $\lambda\in (\bar{\lambda}, 0)$, let us introduce the comparison function
  \begin{equation*}
    \Phi_\lambda(x):=u(x^\lambda)-u(x),\ \ \ x\in \mathcal{H}_\lambda,
  \end{equation*}
  which is well-defined by virtue of the Steiner symmetry of $D$. Our strategy is to show the following properties for $\Phi_\lambda$ for all $\lambda\in(\bar{\lambda}, 0)$:
\begin{align}
\label{201}  \Phi_\lambda> 0 &\ \ \text{in}\ \Sigma_\lambda=\mathcal{H}_\lambda\backslash  \mathcal{R}_\lambda(\overline{U_\varepsilon}) , \\
\label{202}  \partial_{x_1} \Phi_\lambda <0 &\ \ \text{on}\ \Gamma_\lambda\cap D.
\end{align}
This will be done by an initial step for $\lambda\in (\bar{\lambda}, \bar{\lambda}+\tau)$ with $\tau>0$ small and by a continuation step for all $\lambda\in(\bar{\lambda}, 0)$.

 Let us start with the \textbf{initial step}: Recalling \eqref{199}, we have that
  \begin{equation*}
    \Delta \Phi_\lambda+c_\lambda \Phi_\lambda =0\ \ \ \text{in}\ \overline{ \Sigma_\lambda },
  \end{equation*}
  where
  \begin{equation}\label{204}
    c_\lambda(x)=\begin{cases}
                   \displaystyle\frac{f(u(x^\lambda))-f(u(x))}{u(x^\lambda)-u(x)}, & \mbox{if}\ \ u(x^\lambda)\not=u(x), \\
                   0, & \mbox{if}\  \ u(x^\lambda)=u(x).
                 \end{cases}
  \end{equation}
 We check that $c_\lambda\in L^\infty(\Sigma_\lambda)$. Notice that $\Phi_\lambda \ge\not\equiv 0$ on the boundary of each connected component of $\Sigma_\lambda$. For $\lambda$ larger than $\bar{\lambda}$ and close to $\bar{\lambda}$, it follows for instance from the maximum principle in sets with bounded diameter and small Lebesgue measure, and then from the strong maximum principle (see, e.g., \cite{Fra}), that $\Phi_\lambda<0$ in $\Sigma_\lambda$. Furthermore, considering that $\Phi_\lambda(x)=0$ on $\Gamma_\lambda\cap D$, by the Hopf lemma, we observe that $\partial_{x_1}\Phi_\lambda(x)=-2\partial_{x_1}u(x)<0$ for all $x\in\Gamma_\lambda\cap D$.

\textbf{Continuation step}: By the initial step the following quantity $\lambda_*$ is well defined
\begin{equation*}
  \lambda_*=\sup\{\lambda\in (\bar{\lambda}, 0):\, \Phi_{\lambda'}>0\ \text{in}\ \Sigma_{\lambda'}\ \text{for all}\ \lambda'\in(\bar{\lambda}, \lambda)\}.
\end{equation*}
Moreover, by the Hopf lemma, we see that both \eqref{201} and \eqref{202} hold for all $\lambda\in (\bar{\lambda}, \lambda_*)$. Our intention is to show $\lambda_*=0$. Suppose $\bar{\lambda}<\lambda_*<0$. Notice first that $\Phi_{\lambda_*}\ge 0$ in $\Sigma_{\lambda_*}$ by continuity. On the other hand, $\Phi_{\lambda_*} \ge\not\equiv 0$ on the boundary of each connected component of $\Sigma_{\lambda_*}$. From \cite{Fra}, there exists $\delta>0$ such that the weak maximum principle holds in any open set $V'\subset V_\varepsilon$ for the solutions $\Phi\in C^2(V')\cap C(\overline{V'})$ of $\Delta \Phi+c\Phi\le 0$ in $V'$ with $\Phi\ge 0$ on $\partial V'$ and $\|c\|_{L^\infty(V')}\le M$, as soon as $\text{meas}(V')\le \delta$. Let $K$ be a compact subset of $\Sigma_{\lambda_*}$ such that
\begin{equation*}
  \text{meas}(\Sigma_{\lambda_*}\backslash K)<\frac{\delta}{2}.
\end{equation*}
Since $\min_{K}\Phi_{\lambda_*}>0$, it follows that there exists $\tilde{\lambda}\in (\lambda_*, 0)$ such that, for all $\lambda\in [\lambda_*, \tilde{\lambda}]$,
\begin{equation*}
  \min_K\Phi_{\lambda}>0,\ \ \ \partial (\Sigma_{\lambda}\backslash K)=\partial \Sigma_{\lambda} \cup \partial K\ \ \ \text{and}\ \ \ \text{meas}(\Sigma_{\lambda}\backslash K)<\delta.
\end{equation*}
For any such $\lambda\in[\lambda_*, \tilde{\lambda}]$, we have that $\Phi_\lambda \ge\not\equiv 0$ on the boundary of each connected component of $\Sigma_{\lambda}\backslash K$. Therefore $\Phi_\lambda \ge 0$ in $\Sigma_{\lambda}\backslash K$, and hence $\Phi_\lambda \ge 0$ in $\Sigma_\lambda$. By further applying the strong maximum principle, we obtain that $\Phi_\lambda > 0$ in $\Sigma_\lambda$. This contradicts the definition of $\lambda_*$. As a conclusion, $\lambda_*=0$ and \eqref{201} and \eqref{202} both hold for all $\lambda\in (\bar{\lambda},0)$. In particular, we have $\partial_{x_1} u>0$ in $\{x\in D: x_1<0\}$, and the proof is thus complete.

\begin{remark}
We note that the above proof can also be applied, with minor modifications, to domains with lower regularity, such as piecewise smooth domains. In particular, if the domain is enclosed by a regular polygon, then the unique interior stagnation point of steady flows must be the center of the regular polygon.
\end{remark}

\subsection*{Data Availability} Data sharing is not applicable to this article as no datasets were generated or analyzed during the current study.

\subsection*{Declarations}

\smallskip
\ \ \\

\noindent \textbf{Conflict of interest} The authors declare that they have no conflict of interest.


\bigskip

\bigskip


	\phantom{s}
	\thispagestyle{empty}


\begin{thebibliography}{99}

\bibitem{Ago}
V. Agostiniani, S. Borghini and L. Mazzieri, On the Serrin problem for ring-shaped domains, \textit{J. Eur. Math. Soc.}, 2024, published online first, DOI 10.4171/JEMS/1422 (arXiv:2109.11255).

\bibitem{Ale}
A. D. Alexandrov, A characteristic property of the spheres, \textit{Ann. Mat. Pura Appl.}, 58 (1962), 303--354.



\bibitem{Bro0}
 F. Brock, Continuous Steiner-symmetrization, \textit{Math. Nachr.}, 172 (1995), 25--48.


\bibitem{Bro1}
 F. Brock, Continuous rearrangement and symmetry of solutions of elliptic problems, \textit{Proc. Indian Acad. Sci. Math. Sci.}, 110 (2000), no. 2, 157--204.

\bibitem{Bro3}
F. Brock, Rearrangements and applications to symmetry problems in PDE, in: Handbook of Differential Equations: Stationary Partial Differential Equations, Volume IV, pp. 1--60, Elsevier/North-Holland, Amsterdam, 2007.

\bibitem{Bro2}
F. Brock, Symmetry for a general class of overdetermined elliptic problems, \textit{NoDEA Nonlinear Differential Equations Appl.}, 23 (2016), no. 3, Art. 36, 16 pp.


\bibitem{Car}
J. A. Carrillo, S. Hittmeir, B. Volzone and Y. Yao, Nonlinear aggregation-diffusion equations: radial symmetry and long time asymptotics, \textit{Invent. Math.}, 218 (2019), no. 3, 889--977.


\bibitem{Cons}
P. Constantin, T. D. Drivas and D. Ginsberg, Flexibility and rigidity in steady fluid motion, \textit{Comm. Math. Phys.}, 385 (2021), no. 1, 521--563.

\bibitem{Coti}
M. Coti Zelati, T. M. Elgindi and K. Widmayer, Stationary structures near the Kolmogorov and Poiseuille flows in the 2d Euler equations, \textit{Arch. Ration. Mech. Anal.}, 247, 12 (2023). https://doi.org/10.1007/s00205-023-01842-3.

\bibitem{Day}
M. A. Day, The no-slip condition of fluid dynamics, \textit{Erkenntnis}, 33 (1990), no. 3, 285–296.

\bibitem{Dri}
T. D. Drivas and T. M. Elgindi, Singularity formation in the incompressible Euler equation in finite and infinite time, \textit{EMS Surv. Math. Sci.}, 10 (2023), no. 1, 1--100.


\bibitem{En1}
A. Enciso, A. J. Fernández and D. Ruiz, Smooth nonradial stationary Euler flows on the plane with compact support, arXiv:2406.04414.

\bibitem{En2}
A. Enciso, A. J. Fernández, D. Ruiz and P. Sicbaldi, A Schiffer-type problem for annuli with applications to stationary planar Euler flows, arXiv:2309.07977.

\bibitem{Evans}
L. C. Evans and R. F. Gariepy, Measure Theory and Fine Properties of Functions. Textbooks in Mathematics, CRC Press, Boca Raton (2015).

\bibitem{Far}
A. Farina and E. Valdinoci, On partially and globally overdetermined problems of elliptic type, \textit{Amer. J. Math.}, 135 (2013), no. 6, 1699--1726.

\bibitem{Fra}
L. E. Fraenkel, An introduction to maximum principles and symmetry in elliptic problems. Cambridge Tracts in Mathematics, 128. Cambridge University Press, Cambridge, 2000. x+340 pp. ISBN: 0-521-46195-2.

\bibitem{Gal}
G. P. Galdi, An introduction to the mathematical theory of the Navier-Stokes equations. Steady-state problems. Second edition. Springer Monographs in Mathematics. Springer, New York, 2011. xiv+1018 pp.

\bibitem{Gom}
J. Gómez-Serrano, J. Park, J. Shi and Y. Yao, Symmetry in stationary and uniformly rotating solutions of active scalar equations, \textit{Duke Math. J.}, 170 (2021), no. 13, 2957--3038.

\bibitem{Gui}
C. Gui, C. Xie and H. Xu, On a classification of steady solutions to two-dimensional Euler equations, arXiv:2405.15327.

\bibitem{HK}
F. Hamel and A. Karakhanyan, Potential flows away from stagnation in infinite cylinders, 2023, hal-04320798.

    \bibitem{HN3}
F. Hamel and N. Nadirashvili, Shear flows of an ideal fluid and elliptic equations in unbounded domains, \textit{Comm. Pure Appl. Math.}, 70 (2017), no. 3, 590–608.

    \bibitem{HN2}
F. Hamel and N. Nadirashvili, Parallel and circular flows for the two-dimensional Euler equations. Semin. Laurent Schwartz EDP Appl. 2017–2018, exp. V, 1--13.

    \bibitem{HN1}
    F. Hamel and N. Nadirashvili, A Liouville theorem for the Euler equations in the plane, \textit{Arch. Ration. Mech. Anal.}, 233 (2019), no. 2, 599--642.



    \bibitem{HN}
    F. Hamel and N. Nadirashvili, Circular flows for the Euler equations in two-dimensional annular domains, and related free boundary problems, \textit{J. Eur. Math. Soc. (JEMS)}, 25 (2023), no. 1, 323--368.

\bibitem{Kam}
N. Kamburov and L. Sciaraffia, Nontrivial solutions to Serrin's problem in annular domains, \textit{Ann. Inst. H. Poincaré C Anal. Non Linéaire}, 38 (2021), no. 1, 1--22.


\bibitem{Kaw}
B. Kawohl, Rearrangements and convexity of level sets in PDE, vol. 1150. Springer Lecture Notes, New York (1985).

\bibitem{Kor}
M. V. Korobkov, K. Pileckas and R. Russo, Solution of Leray's problem for stationary Navier-Stokes equations in plane and axially symmetric spatial domains, \textit{ Ann. of Math.}, (2) 181 (2015), no. 2, 769--807.


\bibitem{Li}
C. Li, Y. Lv, H. Shahgholian and C. Xie, Analysis on the steady Euler flows with stagnation points in an infinitely long nozzle, arXiv:2203.08375v3.


\bibitem{Lieb}
E. H. Lieb and M. Loss, Analysis, 2nd ed., Graduate Studies in Mathematics, vol. 14, American Mathematical Society, Providence, RI, 2001.

\bibitem{Lions}
P.-L. Lions, Mathematical Topics in Fluid Dynamics. Incompressible Models, vol. 1. Oxford Science Publication, Oxford (1996).

\bibitem{Nad}
N. Nadirashvili, On stationary solutions of two-dimensional Euler equation, \textit{Arch. Ration. Mech. Anal.}, 209 (2013), 729--745.

\bibitem{Ni}
W.-M. Ni, Qualitative properties of solutions to elliptic problems, Stationary partial differential equations. Vol. I, Handb. Differ. Equ., North-Holland, Amsterdam, 2004, pp. 157--233.


\bibitem{Nit}
C. Nitsch and C. Trombetti, The classical overdetermined Serrin problem, \textit{Complex Var. Elliptic Equ.}, 63 (2018), no. 7--8, 1107--1122.

\bibitem{Rei}
W. Reichel, Radial symmetry by moving planes for semilinear elliptic boundary value problems on annuli and other nonconvex domains. In: Bandle, C., et al. (eds.) Progress in partial differential equations, elliptic and parabolic problems. Pitman Research Notes, vol. 325 (1995), 164--182.

\bibitem{Ruiz}
D. Ruiz, Symmetry results for compactly supported steady solutions of the 2D Euler equations, \textit{Arch. Ration. Mech. Anal.}, 247 (2023), no. 3, Paper No. 40, 25 pp.


\bibitem{Ruiz1}
D. Ruiz, Nonsymmetric sign-changing solutions to overdetermined elliptic problems in bounded domains, arXiv:2211.14014v2.



\bibitem{Se}
J. Serrin, A symmetry problem in potential theory, \textit{Arch. Ration. Mech. Anal.}, 43 (1971), 304--318.

\bibitem{Si}
B. Sirakov, Symmetry for exterior elliptic problems and two conjectures in potential theory, \textit{Ann. Inst. H. Poincaré Anal. Non Linéaire}, 18 (2001), 135--156.


\bibitem{Soly}
A. Yu. Solynin, Exercises on the theme of continuous symmetrization, \textit{Comput. Methods Funct. Theory}, 20 (2020), no. 3-4, 465--509.



	\end{thebibliography}
\end{document}